\documentclass[11pt,a4paper,fleqn]{article}
\usepackage[utf8]{inputenc}
\usepackage[english]{babel}
\usepackage{mathrsfs}
\usepackage{amssymb,amsmath,amsthm,amsfonts}
\usepackage[left=1.5cm,right=1.5cm,top=1.5cm,bottom=1.5cm]{geometry} 
\usepackage{indentfirst}
\usepackage{verbatim}
\usepackage{float}

\numberwithin{equation}{section}

\usepackage{pgf,tikz}
\usetikzlibrary{arrows}
\usepackage{caption}

\usepackage{color}

\vfuzz \hfuzz
\newlength{\saveparindent}
\raggedright
\newtheorem{theorem}{Theorem}[section]
\newtheorem{lemma}[theorem]{Lemma}
\newtheorem{proposition}[theorem]{Proposition}

\newtheorem{definition}[theorem]{Definition}
\newtheorem{corollary}[theorem]{Corollary}

\let\originalleft\left
\let\originalright\right
\renewcommand{\left}{\mathopen{}\mathclose\bgroup\originalleft}
\renewcommand{\right}{\aftergroup\egroup\originalright}

\newcommand{\llav}[1]{  \left\{#1\right\} }
\newcommand{\pic}[1]{  \left\langle #1\right\rangle }
\newcommand{\conv}[1]{\xrightarrow[ \hspace{0.15cm} #1  \hspace{0.12cm}]{}}
\newcommand{\norm}[1]{  \left\|#1\right\| }
\newcommand{\pare}[1]{\left(#1\right)}
\newcommand{\corch}[1]{  \left[#1\right] }
\newcommand{\abs}[1]{  \left|#1\right| }
\newcommand{\CAL}[1]{\mathcal{#1}}  
\newcommand{\BB}[1]{\mathbb{#1}}
\newcommand{\SCR}[1]{\mathscr{#1}}

\newcommand{\tn}[1]{\textnormal{#1}}
\newcommand{\ubar}[1]{\text{\b{$#1$}}}
\newcommand{\ul}[1]{\underline{#1}}

\def\mathcolor#1#{\@mathcolor{#1}}
\def\@mathcolor#1#2#3{%
  \protect\leavevmode
  \begingroup
    \color#1{#2}#3%
  \endgroup
}

\newcommand{\blanco}[1]{\mathcolor{white}{#1}}

\def\11{\mbox{\bfseries{1}} }
\def\defi{\mathrel{\mathop:}=}
\def\fide{=\mathrel{\mathop:}}

\def\11{\textnormal{\textbf{1}}}
\def\mm{\hspace{-1.1pt}}
\def\R{\mathfrak R}
\def\I{\mathfrak I}

\DeclareMathOperator{\dr}{\textit{dr}}
\DeclareMathOperator{\ds}{\textit{ds}}
\DeclareMathOperator{\proj}{\textnormal{Pr}}
\DeclareMathOperator{\Dis}{\textnormal{dis}}
\DeclareMathOperator{\tr}{\textnormal{tr}}
\DeclareMathOperator{\MP}{\textnormal{MP}}
\DeclareMathOperator{\Var}{\textnormal{Var}}

\makeatletter
\newread\pin@file
\newcounter{pinlineno}
\newcommand\pin@accu{}
\newcommand\pin@ext{pintmp}
\newcommand*\partialinput [3] {%
  \IfFileExists{#3}{%
    \openin\pin@file #3
    \setcounter{pinlineno}{1}
    \@whilenum\value{pinlineno}<#1 \do{%
      \read\pin@file to\pin@line
      \stepcounter{pinlineno}%
    }
    \addtocounter{pinlineno}{-1}
    \let\pin@accu\empty
    \begingroup
    \endlinechar\newlinechar
    \@whilenum\value{pinlineno}<#2 \do{%
      \readline\pin@file to\pin@line
      \edef\pin@accu{\pin@accu\pin@line}%
      \stepcounter{pinlineno}%
    }
    \closein\pin@file
    \expandafter\endgroup
    \scantokens\expandafter{\pin@accu}%
  }{%
    \errmessage{File `#3' doesn't exist!}%
  }%
}
\makeatother

\title{Turing instability in a model with two interacting Ising lines: linear stability and non-equilibrium fluctuations}
\author{Monia Capanna\thanks{Universit\`a degli Studi dell'Aquila, Via Vetoio, 67100 L'Aquila, Italy. Email:\,{\tt monia.capanna@graduate.univaq.it}} ,
Nahuel Soprano-Loto\thanks{Gran Sasso Science Institute, Viale F. Crispi 7, 67100 L'Aquila, Italy. Email:\,{\tt sopranoloto@gmail.com}}}

\begin{document}
\sloppy
\maketitle
 
\begin{abstract}
\noindent 
This is the second of two articles on the study of a particle system model that exhibits  a Turing instability type effect.
About the hydrodynamic equations obtained in \cite{CSL17a}, we find conditions under which Turing instability occurs around the null equilibrium solution.
In this instability regime:
for long times at which the process is of infinitesimal order, we prove that the non-equilibrium fluctuations around the hydrodynamic limit are Gaussian;
for times converging to the critical one at which the process is of finite order, we prove that the $\pm 1$-Fourier modes are uniformly away from zero.
\end{abstract}

\section{Introduction}

We continue with the study of the particle system model introduced in \cite{CSL17a}.
This  model consists of two discrete lines (or toruses) of Ising spins.
Each line of spins evolves according to a spin-flip dynamic for which the Gibbs measure associated to a Hamiltonian with ferromagnetic interactions given by macroscopic Kac potentials is reversible.
We consider different inverse temperatures $\beta_1$ and $\beta_2$ in each line, and different potentials $\phi_1$ and $\phi_2$ with associated ranges of interactions $\tau_1$ and $\tau_2$.
In addition, the first line acts as an external field with intensity $\lambda$ over the second one, and vice versa with intensity $-\lambda$.
In the mentioned article, it is proven the hydrodynamic convergenge of the magnetization fields to the system of PDE's
\begin{align}
\begin{aligned}\label{hydro1}
\partial_t u_1\pare{t,r}=&-u_1 
\\[5pt] &+\frac{1}{2}\corch{\tanh\pare{\beta_1 u_1 * \phi_{1}+\beta_1\lambda}
+\tanh\pare{\beta_1 u_1 * \phi_{1}-\beta_1\lambda}} 
\\[5pt] & +u_2\frac{1}{2}\corch{\tanh\pare{\beta_1 u_1 * \phi_{1}+\beta_1\lambda}
-\tanh\pare{\beta_1 u_1 * \phi_{1}-\beta_1\lambda}}
\end{aligned}
\end{align}
\begin{align}
\begin{aligned}\label{hydro2}
\partial_t u_2\pare{t,r}=&-u_2 
\\[5pt] &+\frac{1}{2}\corch{\tanh\pare{\beta_2 u_2 * \phi_{2}+\beta_2\lambda}
+\tanh\pare{\beta_2 u_2 * \phi_{2}-\beta_2\lambda}} 
\\[5pt] 
& -u_1\frac{1}{2}\corch{\tanh\pare{\beta_2 u_2 * \phi_{2}+\beta_2\lambda}
-\tanh\pare{\beta_2 u_2 * \phi_{2}-\beta_2\lambda}}.
\end{aligned}
\end{align}
This system has $\pare{
\begin{array}{c}
\textbf{0} \\ \textbf{0}
\end{array}
}$ as equilibrium point ($\textbf{0}$ is the function that vanishes everywhere).

The main goals of the present article are
\begin{enumerate}
\item to study linear stability of the hydrodynamic limit around the equilibrium point,
\item
and to study the fluctuations out of equilibrium of the microscopic system starting with a random perturbation of the equilibrium point.
\end{enumerate}

In the first part, 
we focus on the study of the Turing instability effect, introduced in \cite{Tur52}, of the hydrodynamic limit.
More precisely, we find conditions on the macroscopic parameters
$\llav{\beta_1,\beta_2,\tau_1,\tau_2,\lambda}$
under which Turing instability occurs,
namely the $k$-th Fourier mode of the linearized version of system (\ref{hydro1}-\ref{hydro2}) is linearly stable 
%at the equilibrium point $\pare{
%\begin{array}{c}
%0 \\ 0
%\end{array}
%}$
for $k=0$, 
and linearly unstable for some other value of the wave number $k$ (linear stability will be defined properly later).
Furthermore, conditions are found for which the instability occurs only for $k=\pm 1$, the so called unimodular Turing instablity.
This is the content of Theorem \ref{hi77}.

%The proofs of these facts are elementary and rely on the perturbation of the critical parameter $\lambda^*$ at which the determinant of the matrix associated to the $0$-Fourier mode of the linearized version of the hydrodynamic equation vanishes. 

In the second part,
we start the process with a microscopic random perturbation of the macroscopic equilibrium point and prove results for two time scalings:
\begin{enumerate}
\item
at a time that scales as $\log\gamma^{-\theta}$ ---$\gamma^{-1}$ the number particles,
$\theta\in\pare{0,1}$---,
time at which the process is of infinitesimal order $\gamma^{\frac{1}{2}-\frac{\theta}{2}}$, we prove that, under the proper scaling, the process has Gaussian limiting distribution;
\item
at a time that converges to the critical time at which the process is of finite order,
critical time that scales as $\log \gamma^{-1}$,
we prove that the $\pm 1$-Fourier modes are uniformly away from $0$ with probability that goes to $1$ as the number of particles goes to $\infty$.
\end{enumerate}
These are the contents of Theorems \ref{1} and \ref{2}.
The proofs rely on the approximation of the original process by its linearized version by the use of a modified argument that is classical from differential equations and numerical analysis; once this approximation is done, the theorems follow from well known results of convergence of martingales.

\section{Definition and statement of the results}

We briefly recall the definition of the model introduced in \cite{CSL17a}.
Consider the unit (macroscopic) torus $\BB T$, that we identify with the real interval $\left[0,1\right)$.
The microscopic torus $\Lambda_\gamma$ is defined as $\pare{\gamma^{-1}\BB T}\cap \BB Z$,  $\gamma^{-1}\in\BB N$.
Elements of $\Lambda_\gamma$ are denoted by the letters $x$ and $y$.
For every $\gamma$, we define a continuous time Markov process 
$\pare{\ul\sigma_\gamma\pare{t}}_{t\ge 0}=
\pare{\pare{\sigma_{\gamma,1}\pare{t},\sigma_{\gamma,2}\pare{t}}}_{t\ge 0}$ with state space
$\llav{-1,1}^{\Lambda_\gamma}\times\llav{-1,1}^{\Lambda_\gamma}$.
%All these processes are defined in the same abstract probability space $\pare{\Omega,\CAL F,\BB P}$.
%For $t\ge 0$ and $i\in\llav{1,2}$, $\sigma_{\gamma,i}\pare{t,x}\in\llav{-1,1}$ denotes the spin of $\sigma_{\gamma,i}\pare{t}$ at the site $x$.
%The initial distribution is given in terms of a pair of functions $\psi_1,\psi_2\in C\pare{\BB T,\corch{-1,1}}$ as follows:
%the family $\llav{\sigma_{\gamma,i}\pare{0,x}:x\in\Lambda_\gamma, i\in\llav{1,2}}$ is independent, and
%$\BB E\pare{\sigma_{\gamma,i}\pare{0,x}}=\psi_i\pare{\gamma x}$ for every $i\in\llav{1,2}$ and  $x\in\Lambda_\gamma$.
Unlike \cite{CSL17a}, the initial conditions will not be general here.
Instead, we will consider independent centered random spins, namely the family $\llav{\sigma_{\gamma,i}\pare{0,x}:i\in\llav{1,2},x\in\Lambda_\gamma}$ is independent and $\BB P\pare{\sigma_{\gamma,i}\pare{0,x}=\pm 1}=\frac{1}{2}$ for every $x\in\Lambda_\gamma$ and every $i\in\llav{1,2}$.
In the notation of \cite{CSL17a}, this means $\psi_1=\psi_2\equiv 0$.
In \cite{CSL17a}, we considered general Kac kernels;
here we  consider particular ones.
For $i\in\llav{1,2}$, let $\phi_i\pare{r,\tilde r}\defi \sum_{a\in\BB Z}\tilde\phi_i\pare{r,\tilde r+a}$, where $\tilde\phi_i:\BB R\times\BB R\to \pare{0,\infty}$ is the Gaussian kernel with variance $\tau_i>0$ defined by
\begin{align}\nonumber
\tilde \phi_i\pare{r,r'}\defi \frac{1}{\sqrt{2\pi \tau_i}}e^{-\frac{\pare{r- r'}^2}{2\tau_i}}
\end{align}
($\phi_i$ is the periodized version of $\tilde\phi_i$).
Recall the definition of the discrete convolution:
\begin{align}\nonumber
\pare{\sigma_\gamma *\phi_i}\pare{x}\defi \gamma\sum_{y\in\Lambda_\gamma}\sigma_\gamma\pare{y}\phi_i\pare{\gamma x,\gamma y}.
\end{align}
As before, we have the inverse temperature $\beta_i>0$ associated to the $i$-th line,
and a parameter $\lambda>0$ that describes the interaction between the two lines.
The the spin-flip rates of the generator of our Markov process are given by
\begin{align}\nonumber
R_{1}\pare{x,\ul\sigma_\gamma}
=\frac{\exp\llav{-\beta_1\sigma_{\gamma,1}(x)
\pare{\sigma_{\gamma,1}*\phi_1}\pare{x}}
\exp\llav{-\beta_1\lambda\sigma_{\gamma,1}(x)
 \sigma_{\gamma,2}(x)}
}
{2\cosh  \llav{
\beta_1\pare{\sigma_{\gamma,1}*\phi_1}\pare{x}+\beta_1\lambda\sigma_{\gamma,2}\pare{x}}},
\end{align}
and
\begin{align}\nonumber
R_{2}\pare{x,\ul\sigma_{\gamma}}
=\frac{\exp\llav{-\beta_2\sigma_{\gamma,2}(x)
\pare{\sigma_{\gamma,2}*\phi_2}\pare{x}}
\exp\llav{\beta_2\lambda\sigma_{\gamma,2}(x)
 \sigma_{\gamma,1}(x)}
}
{2\cosh  \llav{
\beta_2\pare{\sigma_{\gamma,2}*\phi_2}\pare{x}-\beta_2\lambda\sigma_{\gamma,1}\pare{x}}}.
\end{align}

If, instead of random initial configurations, we consider deterministic ones with vanishing integral against every continuous function,  Theorem 2.1 of \cite{CSL17a} still holds (with the same proof).
In particular,  if $\eta_\gamma\pare{t,x}=\sigma_{\gamma,1}\pare{t,x}\sigma_{\gamma,2}\pare{t,x}$ is the correlation field, the following corollary holds.

\begin{corollary}\label{arjona}
Consider a sequence $\pare{\ul\sigma_{\gamma}\pare{0}}_{\gamma}$ of deterministic initial configurations satisfying $\pic{\sigma_{\gamma,i}\pare{0},G}\conv{\gamma\to 0}0$ for every $G\in C\pare{\BB T,\BB C}$ and $i\in\llav{1,2}$.
Then, for every $T>0$ and $G\in C\pare{\BB T,\BB C}$,
the limits
\begin{align}\nonumber
\begin{aligned}
&\sup_{0\le t\le T}\abs{\pic{\sigma_{\gamma,i}\pare{t},G}}\conv{\gamma\to 0}0, \ i\in\llav{1,2}
\\[5pt]
&\sup_{0\le t\le T}\abs{\pic{\eta_{\gamma}\pare{t},G}}\conv{\gamma\to 0}0
\end{aligned}
\end{align}
hold in $\BB P$-probability.
%where $u_1,u_2:[0,\infty)\times \BB T \to \BB R$ are the solutions of the system of partial differential equations
%\begin{align}
%\begin{aligned}\label{hydro3}
%\partial_t v\pare{t,r}=&-2v
%\\[5pt]  &+\frac{1}{2}\corch{\tanh\pare{\beta_1 u_1 * \phi_{1}+\beta_1\lambda}
%-\tanh\pare{\beta_1 u_1 * \phi_{1}-\beta_1\lambda}}
%\\[5pt] 
%&-\frac{1}{2}\corch{\tanh\pare{\beta_2u_2 * \phi_{2}+\beta_2\lambda}
%+\tanh\pare{\beta_2 u_2 * \phi_{2}-\beta_2\lambda}}
%\\[5pt]
%& +\frac{u_1}{2}\corch{\tanh\pare{\beta_2 u_2 * \phi_{2}+\beta_2\lambda}
%+\tanh\pare{\beta_2 u_2 * \phi_{2}-\beta_2\lambda}}
%\\[5pt] 
% & +\frac{u_2}{2}\corch{\tanh\pare{\beta_1 u_1 * \phi_{1}+\beta_1\lambda}
%+\tanh\pare{\beta_1 u_1 * \phi_{1}-\beta_1\lambda}}
%\end{aligned}
%\end{align}
%with initial condition $u_i\pare{0,\cdot}\equiv 0$, $i\in\llav{1,2}$.
\end{corollary}

Observe that we are taking $\BB C$ instead of $\BB R$ as the codomain of the test function $G$.
This is because our relevant test functions will be elements of the complex Fourier base $\pare{F^{\pare{k}}}_{k\in\BB Z}$, with $F^{\pare{k}}$ defined as
\begin{align}
F^{\pare{k}}\pare{r}\defi e^{2\pi\texttt{i} k r}.
\end{align}
(Observe that we use a different font for $\texttt{i}=\sqrt{-1}$).
This is an orthonormal  basis of $L^2\pare{\BB T,\BB C}$ with the inner product defined by 
\begin{align}
\pic{G_1,G_2}\defi
\int_0^1 G_1\pare{r}\overline{G_2\pare{r}}dr
\end{align}
(the overline is for the complex conjugate).
We mention here that 
$\hat\phi_i^{\pare{k}}\defi \pic{\phi_i\pare{0,\cdot},F^{\pare{k}}}=e^{-2\pi^2 \tau_i k^2}$ for every $k\in\BB Z$ and $i\in\llav{1,2}$.
If we restrict the domain to $\gamma\Lambda_\gamma$, the family $\llav{F^{\pare{k}}:k=0,\ldots,\gamma^{-1}-1}$ is an orthonormal basis of $L^2\pare{\gamma\Lambda_\gamma,\BB C}	$ under the inner product
\begin{align}\nonumber
\pic{G_1,G_2}_\gamma\defi\gamma\sum_{x\in\Lambda_\gamma}G_1\pare{\gamma x}\overline{G_2\pare{\gamma x}}.
\end{align}
In particular, for a spin configuration $\sigma_\gamma\in\llav{-1,1}^{\Lambda_\gamma}$, we have
\begin{align}
\sigma_\gamma\pare{x}=\sum_{k=0}^{\gamma^{-1}-1}\pic{\sigma_\gamma,F^{\pare{k}}}_\gamma \, F^{\pare{k}}\pare{\gamma x}
\end{align}
for every $x\in\Lambda_\gamma$.
See \cite{Ter99} for a detailed presentation of discrete Fourier analysis.

\subsection{Linear stability of the hydrodynamic equations and Turing instability}

In this subsection, we study linear stability of the system of equations (\ref{hydro1}-\ref{hydro2}).
Using the approximation
\begin{align}\nonumber
\tanh\pare{\lambda\beta_i\pm \beta_1 u_i * \phi_i}\approx
\tanh\pare{\lambda\beta_i}\pm 
\frac{\beta_i}{\corch{\cosh\pare{\lambda\beta_i}}^2} u_i * \phi_i,
\end{align}
we obtain the following linearized version of the hydrodynamic equations:
\begin{align}\nonumber
\frac{d}{dt}u_1=&-u_1
+\frac{\beta_1}{\corch{\cosh(\lambda\beta_1)}^2} \,  u_1 * \phi_1
+u_2\tanh\pare{\lambda\beta_1}
\\[0.5cm]
\frac{d}{dt}u_2=&-u_2
+\frac{\beta_2}{\corch{\cosh(\lambda\beta_2)}^2} \,  u_2 * \phi_2
-u_1\tanh\pare{\lambda\beta_2}.
\end{align}
By taking the $k$-th Fourier transform  ---i.e. by applying the transformation
$u\mapsto \pic{u, F^{\pare{k}}}$---,
and by using that the Fourier transform turns convolutions into products, 
we obtain the $k$-th Fourier system
\begin{align}\nonumber
\frac{d}{dt}\pare{
\begin{array}{c}
\hat u_1^{\pare{k}}
\\[5pt]
\hat u_2^{\pare{k}}
\end{array}
}
=
\pare{
\begin{array}{cc}
-1+\alpha_1\hat{\phi}_1(k) & \tanh(\beta_1\lambda)
\\[5pt]
-\tanh(\beta_2\lambda)  &  -1+\alpha_2\hat{\phi}_2(k) 
\end{array}
}
\pare{
\begin{array}{c}
\hat u_1^{\pare{k}}
\\[5pt]
\hat u_2^{\pare{k}}
\end{array}
},
\end{align}
where $\alpha_i\defi \frac{\beta_i}{\corch{\cosh\pare{\lambda \beta_i}}^2}$.
%Here $\hat\phi_i^{\pare{k}}=\int_0^1 \phi_i\pare{0,r}F^{\pare{k}}\pare{r}dr$ is the $k$-th Fourier transform of $\phi_i$.
We put a name to the previous matrix:
\begin{align}\nonumber
A^{\pare{k}}\defi \pare{
\begin{array}{cc}
-1+\alpha_1\hat{\phi}_1^{\pare{k}} & \tanh(\beta_1\lambda)
\\[5pt]
-\tanh(\beta_2\lambda)  &  -1+\alpha_2\hat{\phi}_2^{\pare{k}} 
\end{array}
}.
\end{align}
For every $k$,
the eigenvalues of $A^{\pare{k}}$ are
\begin{align}\nonumber
\mu_1^{\pare{k}}=\frac{1}{2}\tr^{\pare{k}}+\frac{1}{2}\sqrt{\Dis^{\pare{k}}} \  \text{and} \ 
\mu_2^{\pare{k}}=\frac{1}{2}\tr^{\pare{k}}-\frac{1}{2}\sqrt{\Dis^{\pare{k}}},
\end{align}
where $\Dis^{\pare{k}}\defi \corch{\tr^{\pare{k}}}^2-4\det^{\pare{k}}$,
$\tr^{\pare{k}}\defi\tr A^{\pare{k}}$, and $\det^{\pare{k}}\defi\det A^{\pare{k}}$.

\begin{definition}
We say that the $k$-th Fourier mode of the hydrodynamic equations is linearly stable if $\max\llav{\mathfrak{R}\pare{\mu^{(k)}_1},\mathfrak{R}\pare{\mu^{\pare{k}}_2}}<0$; we say that it is linearly unstable if $\max\llav{\mathfrak{R}\pare{\mu^{\pare{k}}_1},\pare{\mathfrak{R}\mu^{\pare{k}}_2}}>0$.
\end{definition}

\begin{definition}[Turing instability] \label{okio} \
\begin{itemize}
\item[] 
We say that Turing instability occurs if (i) the hydrodynamic equations are linearly stable for $k=0$, and  (ii) are linearly unstable for some other $k$.
\item[]
We say that unimodular Turing instability occurs if the hydrodynamic equations are linearly unstable for $k=\pm 1$, and linearly stable for $k\neq \pm 1$ 
---i.e. if the Turing instability occurs and if (ii) holds only for $k=\pm 1$.
\end{itemize}
\end{definition}

%\blue{(As the determinant and the trace does not change if we take $-\lambda$ instead of $\lambda$, it is not important while analyzing Turing instability which one is the predator.)}

From now on, we suppose $\beta_1\ge \beta_2$.
This does not take away generality while analyzing linear stability.
Indeed, $\tr^{\pare{k}}$ and $\det^{\pare{k}}$ remain invariant if we switch the role of the pairs $\pare{\beta_1,\phi_1}$ and $\pare{\beta_2,\phi_2}$, so the eigenvalues do not change either.
Observe also that $A^{\pare{k}}=A^{\pare{-k}}$ for every $k$.

%We will analyze the particular case in which the Kac potentials are periodic Gaussian.
%More precisely, let $\dis{\phi_i\pare{r,\tilde r}\defi \sum_{a\in\BB Z}\tilde\phi_i\pare{r,\tilde r+a}}$, where $\tilde\phi_i:\BB R\times\BB R\to \pare{0,\infty}$ is the Gaussian kernel with variance $\tau_i>0$:
%\begin{align}\nonumber
%\tilde \phi_i\pare{r,r'}\defi \frac{1}{\sqrt{2\pi \tau_i}}e^{-\frac{\pare{r- r'}^2}{2\tau_i}}.
%\end{align}
%An easy calculation shows that, in this case, $\hat\phi_i^{\pare{k}}=e^{-2\pi^2 \tau_i k^2}=e^{-\tilde \tau_i k^2}$ for every $k$, using the notation $\tilde \tau_i=2\pi^2 \tau_i$.

\begin{theorem}\label{hi77} \
\begin{enumerate}
\item[(i)] 
Turing instability can only occur in the presence of one of the following situations:
\begin{itemize}
\item[] $\alpha_2<1<\alpha_1$ and $\tau_1<\tau_2$ or
\item[] $\alpha_1<1<\alpha_2$ and $\tau_2<\tau_1$.
\end{itemize} 
\item[(ii)]
For $\beta_1>1$, $\beta_2<1$ and $\beta_1+\beta_2<2$,
we can chose values of $\lambda$, $\tau_1$ and $\tau_2$ such that Turing instability occurs.
\item[(iii)]\label{item3}
Suppose the hypotheses of the previous item hold and suppose in addition that $\beta_2$ is sufficiently close to $1$.
Then the parameters $\lambda$, $\tau_1$ and $\tau_2$ can be chosen in such a way that unimodular Turing instability occurs.
\end{enumerate}
\end{theorem}

Observe that, in item \emph{(i)}, as $\alpha_i$ is not monotone as a function of $\beta_i$, assumption $\beta_1\ge \beta_2$ does not guarantee $\alpha_1\ge \alpha_2$.

\subsection{Instability at the microscopic level: fluctuations out of equilibrium}

For $k\in\BB Z$ and $i\in\llav{1,2}$, let $X_{\gamma,i}^{\pare{k}}\pare{t}\defi\pic{\sigma_{\gamma,i}\pare{t},F^{\pare{k}}}_\gamma=\gamma\sum_x\sigma_{\gamma,i}\pare{t,x}e^{-2\pi\mathtt i k\gamma x}$,
and let $\ul X_{\gamma}^{\pare{k}}\pare{t}\defi\pare{
\begin{array}{c}
X_{\gamma,1}^{\pare{k}}\pare{t} \\[4pt]
X_{\gamma,2}^{\pare{k}}\pare{t}
\end{array}
}$.
Our object of study will be the sequence of discrete Fourier modes, namely the stochastic element $\pare{\ul X_{\gamma}^{\pare{k}}\pare{t}}_{k\in\BB Z}$.
We will stand over hypotheses (\ref{cond1}-\ref{cond2}) for $k\neq \pm 1$, and \eqref{cond3} for $k=\pm 1$.
Observe that they are stronger that the unimodular instability obtained in item \textit{(iii)} of Theorem \ref{item3}.
In particular, these hypotheses imply that $\mu_2^{\pare{\pm 1}}<0<\mu_1^{\pare{\pm 1}}$, so $\BB C^2$ is a direct sum of their associated eigenvectors.
From now on, we will call $\mu=\mu_{1}^{\pare{\pm 1}}$ the unique positive eigenvalues.
For $\ul z\in\BB C^2$, let 
\begin{align} \label{decomposition}
\ul z=\proj_1\pare{\ul z}+\proj_2\pare{\ul z}
\end{align}
be the unique decomposition in the mentioned eigenspaces (for $i\in\llav{1,2}$, $\proj_i\pare{\ul z}$ is an eigenvector associated to $\mu_i^{\pare{\pm 1}}$).
%Observe that the linear transformations $\proj_1$ and $\proj_2$ depend only on the macroscopic parameters.

%For $k=\pm 1$, $\llav{\ul v_1^{\pare{k}},\ul v_2^{\pare{k}}}$ is a base of the $\BB C$-vector space $\BB C^2$,
%so every $\ul z\in\BB C^2$ can be written in a unique way as
%\begin{align}\label{decomposition}
%\ul z=\xi_1^{\pare{k}}\hspace{-1.1pt}\pare{\ul z} \, \ul v^{\pare{k}}_1+
%\xi_2^{\pare{k}}\hspace{-1.1pt} \pare{\ul z}\,\ul v_2^{\pare{k}}.
%\end{align}
%For $i\in\llav{1,2}$, 
%$\xi_i^{\pare{k}}\pare{\ul z}$ can be written as $a^{\pare{k}}_{i,1}z_1+a^{\pare{k}}_{i,2}z_2$ ---a linear combination of the coordinates of $\ul z$--- with $a^{\pare{k}}_{i,1}$ and $a^{\pare{k}}_{i,2}$ depending only on $\ul v^{\pare{k}}_1$ and $\ul v^{\pare{k}}_2$.
%We respectively call $\proj_1\hspace{-1.5pt}\pare{\ul z}=\xi_1\hspace{-1pt}\pare{\ul z}\hspace{0.8pt} \ul v_1$ and $\proj_2\pare{\ul z}=\xi_2\pare{\ul z} \hspace{0.8pt} \ul v_2$ the projections over $\ul v_1$ and $\ul v_2$.

%We also assume that the spins at time $t=0$ are independent and $\BB P\pare{\sigma_{\gamma,i}\pare{0,x}=\pm 1}=\frac{1}{2}$ for every $x\in\Lambda_\gamma$ and $i\in\llav{1,2}$;
%in other words, we take $\psi_1=\psi_2\equiv 0$.

The following two results are about the fluctuations of the process around the equilibrium solution $\pare{\textbf{0},\!\textbf{0}}$ for a time that goes to $\infty$ as $\gamma\to 0$
(here $\textbf{0}$ denotes the null function).
In the first result, at a time at which the process is still infinitesimal, we prove that the fluctuations are Gaussian;
in the second one, at a time that approaches the time at which the process is of order one, we prove that the $\pm 1$-Fourier modes are away from $\pare{0,0}$.
Let us motivate the choice of the involved times.
At time $t=0$, the process is of order $\gamma^{\frac{1}{2}}$.
As $\mu=\mu_1^{\pare{1}}=\mu_1^{\pare{-1}}$ are the only positive eigenvalues, the leading terms will be the ones associated to the $\pm 1$-Fourier modes.
As these modes increase exponentially, the order of the process at time $t>0$ is $\gamma^{\frac{1}{2}}e^{\mu t}$.
For a parameter $\theta\in\corch{0,1}$, let $t_\theta$ be the time at which the process is of order $\gamma^{\frac{1}{2}-\frac{\theta}{2}}$.
In this way, the order of the process is increasing in $\theta$,  it is of order $\gamma^{\frac{1}{2}}$ at $\theta=0$, and of order $1$ at $\theta=1$. 
From identity $\gamma^{\frac{1}{2}}e^{\mu t_\theta}=\gamma^{\frac{1}{2}-\frac{\theta}{2}}$, we get $t_\theta=\frac{1}{2\mu}\log\gamma^{-\theta}$.

%Then, for $\alpha\in\pare{0,\frac{1}{2}}$, $\frac{1}{2\mu}\log\gamma^{2\alpha-1}$ is the time at which the order of the process is $\gamma^\alpha$.
%Let $-\theta\defi 2\alpha-1$ be the exponent appearing inside the logarithm in the definition of $t_\alpha$,
%and let $t_\theta\defi\frac{1}{2\mu}\log\gamma^{-\theta}$.
%As the parameter $\alpha$ from $\frac{1}{2}$ to $0$, the parameter $\theta$ runs from $0$ to $1$.

\begin{theorem}\label{1}
Let $\theta\in\pare{0,1}$.
For $k\neq\pm 1$, $\gamma^{\frac{\theta}{2}-\frac{1}{2}} \underline{X}_\gamma^{(k)}\pare{t_\theta}$ converges in distribution to the $\delta$-measure concentrated in $\pare{\begin{array}{c}
0 \\ 0
\end{array}
}$.
For $k=\pm 1$, $\gamma^{\frac{\theta}{2}-\frac{1}{2}} \underline{X}_\gamma^{(k)}\pare{t_\theta}$ converges in distribution to 
\begin{align}
\proj_1\pare{\ul{U}+\ul{V}},
\end{align}
where $\ul{U}$ and $\ul{V}$ are random elements of $\BB C^2$ defined as
\begin{align}
\ul{U}=\pare{
\begin{array}{c}
U_1^{\R}+\mathtt{i}U_1^{\I} \\[5pt]
U_2^{\R}+\mathtt{i}U_2^{\I} 
\end{array}
}  \ \ \ \ \
\ul{V}=\pare{
\begin{array}{c}
V_1^{\R}+\mathtt{i}V_1^{\I} \\[5pt]
V_2^{\R}+\mathtt{i}V_2^{\I} 
\end{array}
},
\end{align}
with $\llav{U_1^{\R},U_1^{\I},U_2^{\R},U_2^{\I},V_1^{\R},V_1^{\I},V_2^{\R},V_2^{\I}}$ an independent family of centered Gaussian random variables with variances
\begin{align}
&\Var\pare{U_1^{\R}}=\Var\pare{U_1^{\I}}=\Var\pare{U_2^{\R}}=\Var\pare{U_2^{\I}}=\pi
\\[5pt]
&\Var\pare{V_1^{\R}}=\Var\pare{V_1^{\I}}=\frac{\pi}{\mu}-\tanh\pare{\beta_1\lambda}\corch{\tanh\pare{\beta_1\lambda}-\tanh\pare{\beta_2\lambda}}\frac{\pi}{2\mu^2+2\mu}
\\[5pt]
&\Var\pare{V_2^{\R}}=\Var\pare{V_2^{\I}}=\frac{\pi}{\mu}+\tanh\pare{\beta_2\lambda}\corch{\tanh\pare{\beta_1\lambda}-\tanh\pare{\beta_2\lambda}}\frac{\pi}{2\mu^2+2\mu}.
\end{align} 
\end{theorem}

Observe that the limiting distribution in the previous result does not depend on $\theta$.

The next corollary clarifies the notion of pattern formation:
the limiting distribution at a microscopic scale is a periodic non-homogeneous function with uniform phase and random amplitude.

\begin{corollary}\label{mozart}
Let $G\in C^\infty\pare{\BB T, \BB R}$.
Then, for $i\in\llav{1,2}$, $\gamma^{\frac{\theta}{2}-\frac{1}{2}}\pic{\sigma_{\gamma,i}\pare{t_\theta}, G}$ converges in distribution, as $\gamma\to 0$, to $\pic{\sqrt{A_i} \sin\pare{2\pi\cdot+\Phi_i}, G} $, where $A_i$  and $\Phi_i$ respectivelly have gamma and $[0,2\pi)$-uniform distributions.
\end{corollary}

%In the previous statement, the coefficient $\gamma^{\frac{\theta}{2}-\frac{1}{2}}$ in front of $\ul X_\gamma^{\pare{k}}\pare{t_\theta}$ is a normalizing parameter.
The critical time is defined as $t_c\defi t_1=\frac{1}{2\mu}\log\gamma^{-1}$.
We define the time-rescaled process $\ul Y_\gamma\pare{\theta}\defi\ul X_\gamma\pare{t_c\theta}$.
The previous result is then describing the limiting distribution of $\gamma^{\frac{\theta}{2}-\frac{1}{2}}\ul Y_\gamma\pare{\theta}$ for $\theta\in\pare{0,1}$.
Unfortunately, we are not able to give information about the fluctuations at the critical time.
Nevertheless, in the next theorem, we are able to give information about the distribution of the $\pm 1$-Fourier modes for a time that is approaching $t_c$ or, in the rescaled process $\ul Y_\gamma$, that is approaching $\theta=1$.

\begin{theorem}\label{2}
Let $T_\delta\defi\frac{1}{2\mu}\log \delta^{-1}$.
Then, for $k=\pm 1$,
\begin{align}
\lim_{\delta\to 0} \ \liminf_{\gamma\to 0} \ \BB P\pare{\norm{\ul Y_\gamma^{\pare{k}}\pare{1-\frac{T_\delta}{t_c}}}>\delta}=1.
\end{align}
\end{theorem}

From this theorem, we deduce that
\begin{align}\nonumber
\liminf_{\gamma\to 0}P\pare{\norm{\underline Y_\gamma^{\pare{\pm 1}}\pare{1-\frac{T_\delta}{t_c}}}>\delta}\ge
1-\textnormal{Err}\pare{\delta},
\end{align}
where $\textnormal{Err}\pare{\delta}$ is a error that vanishes as $\delta$ converges to $0$.
In other words, with high probability,  pattern formation occurs at a time which converges to the critical one in the sense that the $\pm 1$-Fourier modes are away from $\pare{
\begin{array}{c}
0 \\ 0
\end{array}
}$ uniformly in $\gamma$.

\section{Proofs}

During the proofs, we work with many constants,
about which it is convenient to clarify that, unless explicitly mentioned, they depend only on the macroscopic parameters $\llav{\beta_1,\beta_2,\tau_1,\tau_2,\lambda}$ and on universal notions.
They will never depend on, for instance, $k\in\BB Z$ or $\gamma$.

\subsection{Proof of theorem \ref{1}}

We start by establishing the decomposition
\begin{align}\label{oliver}
\ul{X}_{\gamma}^{\pare{k}}\pare{t}=\ul{X}_{\gamma}^{\pare{k}}\pare{0}
+\int_0^t A^{\pare{k}}\ul X_{\gamma}^{\pare{k}}\pare{s}+\ul E_{\gamma}^{\pare{k}}\pare{s}ds+\ul M_\gamma^{\pare{k}}\pare{t},
\end{align}
where $\ul E_{\gamma}^{\pare{k}}\pare{s}$ and $\ul M_\gamma^{\pare{k}}\pare{t}$ are respectively the non-linear error and the martingale terms.
Replacing $L_\gamma\pic{\sigma_{\gamma,1}\pare{s},G}$ in the definition of $M_{\gamma,1}^{G}$ by the formula of identity (3.1) in  \cite{CSL17a},
taking $G=F^{\pare{k}}$, and using the notation $A_{\gamma,1}\pare{s,x}=\pare{\sigma_{\gamma,1}\pare{s}*\phi_1}\pare{x}$,
we get
\begin{align}
\begin{aligned}
&X^{\pare{k}}_{ \gamma ,1}\pare{t}
\\[5pt]
&\quad =  X^{\pare{k}}_{ \gamma ,1}\pare{0}
\\[5pt]
&\quad\blanco{=}
+\int_0^t-X^{\pare{k}}_{ \gamma ,1}\pare{s} 
\\[5pt]
&\quad\blanco{=+\int_0^t}
+\frac{ \gamma }{2}
\sum_{x\in\Lambda_ \gamma }\sigma_{ \gamma ,2}\pare{s,x}
\llav{\tanh\corch{\beta_1A_{\gamma,1}\pare{s,x}+\lambda\beta_1}
-\tanh\corch{\beta_1 A_{\gamma,1}\pare{s,x}-\lambda\beta_1}}F^{\pare{k}}  \pare{\gamma x}
\\[5pt]
&\quad\blanco{=+\int_0^t}
+\frac{ \gamma }{2}
\sum_{x\in\Lambda_ \gamma }
\llav{\tanh\corch{\beta_1 A_{\gamma,1}\pare{s,x}+\lambda\beta_1}+\tanh\corch{\beta_1 A_{\gamma,1}\pare{s,x}-\lambda\beta_1}}F^{\pare{k}} \pare{\gamma x}ds  
\\[5pt]
&\quad\blanco{=}
+M^{\pare{k}}_{ \gamma ,1}\pare{t}.
\end{aligned}
\end{align}
Making the expansion
\begin{align}
\tanh\corch{\lambda\beta_1\pm \beta_1A_{\gamma,1}\pare{s,x}}=
\tanh\pare{\lambda\beta_1}\pm 
\beta_1A_{\gamma,1}\pare{s,x}
\tanh'\pare{\lambda\beta_1}+E^{\pm}_{ \gamma ,1}\pare{s,x},
\end{align}
we get
\begin{align}
\begin{aligned}
X^{\pare{k}}_{ \gamma ,1}\pare{t} =&  X^{\pare{k}}_{ \gamma ,1}\pare{0}
\\[5pt]
&+\int_0^t-X^{\pare{k}}_{ \gamma ,1}\pare{s} 
\\[5pt]
&\blanco{+\int_0^t}
+\tanh\pare{\lambda\beta_1}X^{\pare{k}}_{ \gamma ,2}\pare{s}
\\[5pt]
&\blanco{+\int_0^t}
+\tanh'\pare{\lambda\beta_1} \gamma  \sum_{x\in\Lambda_ \gamma }F^{\pare{k}}  \pare{ \gamma x }\beta_1A_{\gamma,1}\pare{s,x}
\\[5pt]
&\blanco{+\int_0^t}
+\frac{ \gamma }{2}
\sum_{x\in\Lambda_ \gamma }\sigma_{ \gamma ,2}\pare{s,x}F^{\pare{k}}  \pare{\gamma x}
\corch{E^+_{ \gamma ,1}\pare{s,x}+E^-_{ \gamma ,1}\pare{s,x}}
\\[5pt]
&\blanco{+\int_0^t}
+\frac{ \gamma }{2}
\sum_{x\in\Lambda_ \gamma }F^{\pare{k}}  \pare{\gamma x}
\corch{E^+_{ \gamma ,1}\pare{s,x}-E^-_{ \gamma ,1}\pare{s,x}}ds  
\\[5pt]
&+M^{\pare{k}}_{ \gamma ,1}\pare{t}.
\end{aligned}
\end{align}
We call $\tilde E^{\pare{k}}_{ \gamma ,1}\pare{s}$ the error associated to the Taylor approximation:
\begin{align}
\begin{aligned}\label{error_taylor}
\tilde E^{\pare{k}}_{ \gamma ,1}\pare{s}\defi&
\frac{ \gamma }{2}
\sum_{x\in\Lambda_ \gamma }F^{\pare{k}}  \pare{\gamma x }
\corch{E^+_{ \gamma ,1}\pare{s,x}+E^-_{ \gamma ,1}\pare{s,x}}
\\[5pt]
&+\frac{ \gamma }{2}
\sum_{x\in\Lambda_ \gamma }\sigma_{ \gamma ,2}\pare{s,x}F^{\pare{k}}  \pare{ \gamma x }
\corch{E^+_{ \gamma ,1}\pare{s,x}-E^-_{ \gamma ,1}\pare{s,x}}.
\end{aligned}
\end{align}
The following is the error coming from approximating the integral by a Riemann sum:
\begin{align}\label{error_product}
W^{\pare{k}}_{ \gamma ,1}\pare{s}\defi
\tanh'\pare{\lambda\beta_1}\corch{ \gamma  \sum_{x\in\Lambda_ \gamma }F^{\pare{k}}  \pare{ \gamma x}\beta_1A_{\gamma,1}\pare{s,x}-\beta_1 \hat\phi_1^{\pare{k}}X^{\pare{k}}_{ \gamma ,1}\pare{s}}.
\end{align} 
Under these definitions, we get
\begin{align}
\begin{aligned}
&X^{\pare{k}}_{ \gamma ,1}\pare{t}=X^{\pare{k}}_{ \gamma ,1}\pare{0}
\\[5pt]
&\blanco{X^{\pare{k}}_{ \gamma ,1}\pare{t}=}
+\int_0^t\corch{1-\alpha_1\hat\phi_1\pare{k}}X^{\pare{k}}_{ \gamma ,1}\pare{s}+\tanh\pare{\lambda\beta_1}X_{ \gamma ,2}^{\pare{k}}\pare{s}+W^{\pare{k}}_{ \gamma ,1}\pare{s}+\tilde E^{\pare{k}}_{ \gamma ,1}\pare{s}ds
\\[5pt]
&\blanco{X^{\pare{k}}_{ \gamma ,1}\pare{t}=}
+M_{ \gamma ,1}^{\pare{k}}\pare{t}.
\end{aligned}
\end{align}

Analogously, we get
\begin{align}
\begin{aligned}
&X^{\pare{k}}_{ \gamma ,2}\pare{t}=X^{\pare{k}}_{ \gamma ,2}\pare{0}
\\[5pt]
&\blanco{X^{\pare{k}}_{ \gamma ,2}\pare{t}=}
+\int_0^t -\tanh\pare{\lambda\beta_2}X^{\pare{k}}_{ \gamma ,1}\pare{s}
+\corch{-1+\alpha_2\hat\phi_2\pare{k}}\pare{\lambda\beta_1}X_{ \gamma ,2}^{\pare{k}}\pare{s}+W^{\pare{k}}_{ \gamma ,2}\pare{s}+\tilde E^{\pare{k}}_{ \gamma ,2}\pare{s}ds
\\[5pt]
&\blanco{X^{\pare{k}}_{ \gamma ,1}\pare{t}=}
+M_{ \gamma ,2}^{\pare{k}}\pare{t},
\end{aligned}
\end{align}
with
\begin{align}
\tanh\corch{\lambda\beta_2\pm \beta_2A_{\gamma,2}\pare{s,x}}=
\tanh\pare{\lambda\beta_2}\pm 
\beta_2 A_{\gamma,2}\pare{s,x}
\tanh'\pare{\lambda\beta_2}+E^{\pm}_{ \gamma ,2}\pare{s,x},
\end{align}
\begin{align}
\begin{aligned}
\tilde E^{\pare{k}}_{ \gamma ,2}\pare{s}\defi&
\frac{ \gamma }{2}
\sum_{x\in\Lambda_ \gamma }F^{\pare{k}}  \pare{\gamma x }
\corch{E^+_{ \gamma ,2}\pare{s,x}-E^-_{ \gamma ,2}\pare{s,x}}
\\[5pt]
&-\frac{ \gamma }{2}
\sum_{x\in\Lambda_ \gamma }\sigma_{ \gamma ,2}\pare{s,x}F^{\pare{k}}  \pare{ \gamma x }
\corch{E^+_{ \gamma ,2}\pare{s,x}+E^-_{ \gamma ,2}\pare{s,x}},
\end{aligned}
\end{align}
and
\begin{align}
W^{\pare{k}}_{ \gamma ,2}\pare{s}\defi
\tanh'\pare{\lambda\beta_2}\corch{ \gamma  \sum_{x\in\Lambda_ \gamma }F^{\pare{k}}  \pare{ \gamma x}\beta_2 A_{\gamma,2}\pare{s,x}-\beta_2 \hat\phi_2^{\pare{k}}X^{\pare{k}}_{ \gamma ,2}\pare{s}}.
\end{align} 
Calling $E^{\pare{k}}_{\gamma,i}\pare{s}\defi \tilde E^{\pare{k}}_{\gamma,i}\pare{s}+W^{\pare{k}}_{\gamma,i}\pare{s}$, we obtain the compact formula \eqref{oliver}.

We now proceed in two steps: we first approximate our process by linearization ---namely by removing the error $\ul E_\gamma^{\pare{k}}\pare{s}$---, and then we prove the convergence of the linearized one.

\subsubsection*{Step 1}

Let $\tilde X_\gamma^{\pare{k}}\pare{t}$ be the unique solution of equation
\begin{align}
&d\ul{\tilde X}_\gamma^{\pare{k}}\pare{t}=A^{\pare{k}}\ul{\tilde X}_\gamma^{\pare{k}}\pare{t}dt+d\ul M_\gamma^{\pare{k}}\pare{t}\\
&\ul{\tilde X}_\gamma^{\pare{k}}\pare{0}= \ul X_\gamma^{\pare{k}}\pare{0}.
\end{align}
We want to prove that
\begin{align}\label{tre}
\norm{\gamma^{\frac{\theta}{2}-\frac{1}{2}}\ul X_\gamma^{\pare{k}}\pare{t_\theta}-\gamma^{\frac{\theta}{2}-\frac{1}{2}}\ul{\tilde X}_\gamma^{\pare{k}}\pare{t_\theta}}\xrightarrow[\gamma\to 0]{\BB P}0
\end{align}
for all $k\in\BB Z$.
%Once we have proven \eqref{tre} and that $\gamma^{\frac{\theta}{2}-\frac{1}{2}}\ul {\tilde X}_\gamma
%^{\pare{k}}\pare{t_\theta}$ has a limit in distribution,
%we can conclude that the limiting distribution of $\gamma^{\frac{\theta}{2}-\frac{1}{2}}\ul {X}_\gamma^{\pare{k}}\pare{t_\theta}$ coincides with the one of  $\gamma^{\frac{\theta}{2}-\frac{1}{2}}\ul {\tilde X}_\gamma^{\pare{k}}\pare{t_\theta}$.
By Duhamel's formula, we have
\begin{align}\label{mon}
\ul X_\gamma^{(k)}(t)&=e^{tA^{\pare{k}} }\ul X_\gamma^{(k)}(0)+\int_0^te^{(t-s)A^{\pare{k} } }\ul E_\gamma^{(k)}(s)ds+\int_0^t e^{(t-s)A^{\pare{k} }}\ul{M}^{(k)}_\gamma(ds) \\
\label{duha02}
\ul{\tilde X}_\gamma^{(k)}(t)&=e^{tA^{\pare{k}} }\ul X_\gamma^{(k)}(0)+\int_0^t e^{(t-s)A^{\pare{k} }}\ul{M}^{(k)}_\gamma(ds)
\end{align}
for all $t\geq 0$ and $k\in\BB Z$.
Then \eqref{tre} follows once we  prove that
\begin{align}\label{quattro}
\gamma^{\frac{\theta}{2}-\frac{1}{2}}\norm{\int_0^{t_\theta}e^{\pare{t_\theta-s}A^{\pare{k}}}\ul E_\gamma^{\pare{k}}\pare{s}\ds}\xrightarrow[\gamma\to 0]{\BB P}0
\end{align}
for all $k \in \BB Z$.
Let $C_1$ be the constant of Lemma \eqref{az}.
For $t\geq 0$, we have
\begin{align}\label{cinque}
\norm{\int_0^{t}e^{\pare{t-s}A^{\pare{k}}}\ul E_\gamma^{\pare{k}}\pare{s}\ds}\leq C_1 \int_0^{t}e^{\pare{t-s}\mu} \norm{\ul E_\gamma^{\pare{k}}\pare{s}}\ds
\end{align}
for $k=\pm 1$, and
\begin{align}\label{sei}
\norm{\int_0^{t}e^{\pare{t-s}A^{\pare{k}}}\ul E_\gamma^{\pare{k}}\pare{s}\ds}\leq C_1\int_0^{t}e^{\frac{1}{2}\pare{t-s}\mathfrak{R}\pare{\mu_1^{\pare{k}}}} \norm{\ul E_\gamma^{\pare{k}}\pare{s}}\ds
\end{align}
for $k\neq\pm 1$.
The proof is then reduced to finding a proper bound for $\norm{\ul E_\gamma^{\pare{k}}\pare{t}}$ in the interval $[0, t_\theta]$.
We get \eqref{quattro} once we prove next Proposition. 
%Define $\alpha=\alpha\pare{\theta}:=\frac{1}{2}-\frac{\theta}{2}$.
\begin{proposition}\label{propl}
Let 
\begin{align}
\mathcal E_{\gamma, \theta, D}\defi
\corch{\sup_{\abs{k}\leq \gamma^{-\frac{1}{8}\pare{1-\theta}}}\norm{\ul{E}^{(k)}_\gamma\pare{t}}< D\gamma^{\frac{1}{4}\pare{3+\theta}}e^{2\mu t} \ \  \ \forall t\in [0, t_\theta]}.
\end{align}
There exists $D$ (defined only in terms of the macroscopic parameters) such that 
$\BB P\pare{\mathcal E_{\gamma,\theta, D}}\xrightarrow[\gamma\to 0]{}1$.
\end{proposition}
\begin{proof}
The choice of the constant $D$ will be made a posteriori.
Let $t^*=t^*\pare{D}$ be the stopping time defined as
\begin{align}
t^*\defi\min\llav{t\geq 0: \sup_{\abs{k}\leq \gamma^{-\frac{1}{8}\pare{1-\theta}}}\norm{E^{\pare{k}}\pare{t}}\ge D\gamma^{\frac{1}{4}\pare{3+\theta}}e^{2\mu t}},
\end{align}
with the convention that the minimum of the empty set is $\infty$.
The assertion follows once we find $D$ such that 
\begin{align}\label{undici}
\BB P\pare{t_\theta<t^*}\xrightarrow[\gamma\to 0]{}1.
\end{align}
Under definition
\begin{align}\label{insieme}
\begin{aligned}
\mathscr A_{\gamma, \theta}\defi
\Bigg[  &\sup_{t\in [0,t_\theta]}\norm{\int_0^t e^{\pare{t-s} A^{\pare{k}}}\ul M_\gamma^{\pare{k}}\pare{ds}}\leq k^2 \gamma^{\frac{1}{8}\pare{3+\theta}}\quad\forall k\in\BB Z \setminus\llav{\pm 1},\\[5pt]
& \sup_{t\in [0,t_\theta]}\norm{\int_0^t e^{-sA^{\pare{k}}}\ul M_\gamma^{\pare{k}}\pare{ds}}\leq \gamma^{\frac{1}{8}\pare{3+\theta}}\quad\forall k\in\llav{\pm 1},\\[5pt]
& \norm{\ul{X}_\gamma^{\pare{k}}\pare{0}}\leq k^2 \gamma^{\frac{1}{8}\pare{3+\theta}}\quad\forall k\in\BB Z\Bigg],
\end{aligned}
\end{align}
\eqref{undici} follows once we
\begin{itemize} 
\item[1.] find $D$ such that
\begin{align}\label{plk}
\mathscr A_{\gamma, \theta}\subset \corch{t_\theta<t^*}
\end{align}
for $\gamma$ small enough, and
\item[2.] prove that
$\BB P\pare{\mathscr A_{\gamma, \theta}}\xrightarrow[\gamma\to 0]{}1$.
\end{itemize}
Item $2$ is a consequence of Lemma \eqref{PPP} in the Appendix.
To prove \eqref{plk}, we need the following Lemma.
\begin{lemma}\label{bound1}
Let $\tau\defi \tau_1\wedge\tau_2$.
There exists a  constant $C$ such that, for all $t\geq 0$, 
\begin{align}
\norm{E_\gamma^{\pare{k}}\pare{t}}\leq
C\corch{\gamma k+\pare{\sum_{j\in\BB Z}e^{-{\tau j^2}}\norm{\underline X_\gamma^{\pare{j}}\pare{t}}}^2}.
\end{align}
\end{lemma}
We first prove \eqref{plk},
postponing the proof of Lemma \eqref{bound1}. 
Fix $\omega\in \mathscr A_{\gamma, \theta}$.
Let $g_\gamma\pare{t}\defi D\gamma^{\frac{1}{4}\pare{3+\theta}}e^{2\mu t}$ (observe that $g_\gamma\pare{\cdot}$ coincides with the function used to define $t^*$).
As the case $t^*=\infty$ is trivial,
we only prove the assertion in the case $t^*<\infty$.
For all $t\in [0, t_\theta\wedge t^*]$, by \eqref{mon} and Lemma \ref{az}, we have
\begin{align}\label{as}
\begin{aligned}
\norm{\ul X_\gamma^{\pare{k}}\pare{t}}&\leq C_1e^{\mu t}\pare{\gamma^{\frac{1}{8}\pare{3+\theta}}+\int_0^t e^{-\mu s}g_\gamma\pare{s}\ds+\gamma^{\frac{1}{8}\pare{3+\theta}}}
\\[5pt]
&\leq \pare{C_1+1}e^{\mu t}\gamma^{\frac{1}{8}\pare{3+\theta}}+\frac{C_1D}{\mu}e^{2\mu t}\gamma^{\frac{1}{4}\pare{3+\theta}}
\end{aligned}
\end{align}
for $k=\pm 1$, and
\begin{align}\label{ad}
\begin{aligned}
\norm{\ul X_\gamma^{\pare{k}}\pare{t}}&\leq \pare{C_1+1} k^2 \gamma^{\frac{1}{8}\pare{3+\theta}}+\frac{C_1 D}{\mu} e^{2\mu t}\gamma^{\frac{1}{4}\pare{3+\theta}}
\end{aligned}
\end{align}
for $k\neq\pm 1$ such that $\abs{k}\leq \gamma^{-\frac{1}{8}\pare{1-\theta}}$.
Observe that, in \eqref{as} and \eqref{ad}, we strongly used that $t<t^*$.
If $C_2$ is the constant defined in Lemma \ref{bound1},
we have
\begin{align}\label{18}
\begin{aligned}
&\sup_{k\leq \gamma^{-\frac{1}{8}\pare{1-\theta}}}\norm{\ul{E}_\gamma^{\pare{k}}\pare{t}}
\\[5pt]
&\quad \leq C_2\gamma^{1-\frac{1}{8}\pare{1-\theta}}
+ C_2\pare{e^{-\tau}\pare{\norm{X_\gamma^{\pare{-1}}}+\norm{X_\gamma^{\pare{1}}}}
+\sum_{\substack{j\neq\pm 1
\\[3pt]
\abs{j}\leq \gamma^{-\frac{1}{8}\pare{1-\theta}}}}
e^{-\tau j^2}\norm{X_\gamma^{\pare{k}}}
+\sum_{\abs{j}>\gamma^{-\frac{1}{8}\pare{1-\theta}}}e^{-\tau j^2}}^2
\end{aligned}
\end{align}
for all $t\geq 0$.
It is easily seen that $\sum_{\abs{j}>\gamma^{-\frac{1}{8}\pare{1-\theta}}}e^{-\tau j^2}\leq \gamma^{\frac{1}{8}\pare{3+\theta}}$ for $\gamma$ small enough.
By plugging estimations \eqref{as} and \eqref{ad} into \eqref{18}, we get that, for all $t\in [0, t_\theta\wedge t^*]$,
\begin{align}\label{fgr}
\begin{aligned}
\sup_{\abs{k}\leq \gamma^{-\frac{1}{8}\pare{1-\theta}}}\norm{\ul{E}_\gamma^{\pare{k}}\pare{t}}\leq f_\gamma\pare{t},
\end{aligned}
\end{align}
where $f_\gamma\pare{t}\defi C_3e^{2\mu t}\gamma^{\frac{1}{4}\pare{3+\theta}}+C_4\pare{D}e^{4\mu t}\gamma^{\frac{1}{2}\pare{3+\theta}}$ with $C_4$ is a constant that depends on $D$.
We will choose $D$ properly in order to have $t_\theta<t^*$ for $\gamma$ sufficiently small, which allows to get \eqref{plk}.
For this purpose, we need a brief analysis of the function $f_\gamma-g_\gamma$ in $[0, \infty)$.
It is easy to prove that, for $D>C_3$ and $\gamma$ sufficiently small, $f_\gamma-g_\gamma$ has a unique root, before which it is negative and after which it is positive.
Since
\begin{align}
\begin{aligned}
f_\gamma\pare{t_\theta}=C_3\gamma^{\frac{3}{4}\pare{1-\theta}}+C_4\pare{D}\gamma^{\frac{3}{2}\pare{1-\theta}}
\end{aligned}
\end{align}
and
\begin{align}
g_\gamma\pare{t_\theta}=D\gamma^{\frac{3}{4}\pare{1-\theta}},
\end{align}
under our choice of $D$ and for $\gamma$ sufficiently small, $f_\gamma\pare{t_\theta}-g_\gamma\pare{t_\theta}<0$, while $f_\gamma\pare{t^*}-g_\gamma\pare{t^*}\geq 0$ by \eqref{fgr} and the definition of $t^*$.
This let us conclude that $t_\theta<t^*$ for $\gamma$ small enough, and  our assertion follows.
\end{proof}

%\subsubsection{Proof of Lemma \ref{bound1}}

\begin{proof}[Proof of lemma \ref{bound1}]
We will analyse just the first component of the vector $\underline E_\gamma^{\pare{k}}\pare{t}$ as the analysis of the second one is analogous.
We first prove that there exists a constant $C_1$ such that
\begin{align}\label{et1}
\abs{W^{\pare{k}}_{\gamma,1}\pare{s}}\le  C_1\gamma k.
\end{align}
We can forget about the factor $\beta_1\tanh'\pare{\lambda\beta_1}$ and control the difference
\begin{align}
\abs{ \gamma  \sum_{x\in\Lambda_ \gamma }F^{\pare{k}}  \pare{ \gamma x}A_{\gamma,1}\pare{s,x}-\hat\phi_1^{\pare{k}}X^{\pare{k}}_{ \gamma ,1}\pare{s}}.
\end{align}
It is easy to see that
\begin{align}
\gamma  \sum_{x\in\Lambda_ \gamma }F^{\pare{k}}  \pare{ \gamma x}A_{\gamma,1}\pare{s,x}
=X_{\gamma,1}^{\pare{k}}\pare{s}\gamma\sum_{x\in\Lambda_\gamma}\phi_1\pare{\gamma x}F^{\pare{k}}  \pare{\gamma x}
\end{align}
(the Fourier transform turns convolution into product).
As $\abs{X_{\gamma,1}^{\pare{k}}}\le 1$, we just need to control the difference
\begin{align}
\abs{\gamma\sum_{x\in\Lambda_\gamma}\phi_1\pare{\gamma x}F^{\pare{k}}  \pare{\gamma x}-\int_{\BB T}\phi_1\pare{r}F^{\pare{k}}  \pare{r}dr}.
\end{align}
It only remains to observe that the increments of the function $\phi_1 F^{\pare{k}}$ are bounded by a constant $C_2$ times $k$.

We now prove that there exists a constant $C_3$ such that 
\begin{align}\label{et2}
\abs{\tilde E_{\gamma,1}^{\pare{k}}\pare{s}}\le C_3 \pare{\sum_{j\in\BB Z}e^{-\tau j^2}\abs{X_{\gamma,i}^{\pare{j}}\pare{t}}}^2.
\end{align}
As the second derivative of the hyperbolic tangent is uniformly bounded, we have
\begin{align}
\abs{E_{\gamma,1}^{\pm}\pare{s,x}}\le C_4\pare{\beta_1A_{\gamma,1}\pare{s,x}}^2.
\end{align}
Then, for some constant $C_5$, 
\begin{align}
\begin{split}
\abs{E_{\gamma, 1}^+(t,x)\pm E_{\gamma, 1}^-(t,x)}
&\leq C_5\pare{A_{\gamma,1}\pare{s,x}}^2
\\[5pt]
&=C_5 \pare{\gamma\sum_{y\in\Lambda_\gamma}\phi_1\pare{\gamma y-\gamma x}\sigma_{\gamma, 1}(t,y)}^2
\\[5pt]
&= C_5\pare{\gamma\sum_{y\in\Lambda_\gamma}\sum_{j=0}^{\gamma^{-1}-1}\hat\phi_1^{(j)}e^{\texttt{i}(\gamma y-\gamma x)2\pi j}\sigma_{\gamma, 1}(t,y)}^2
\\[5pt]
&= C_5\pare{\sum_{j=0}^{\gamma^{-1}-1} e^{-\tau_1 j^2}X_{\gamma,1}^{(j)}(t)}^2
\\[5pt]
&\leq C_5\pare{\sum_{j=0}^{\gamma^{-1}-1}e^{-\tau j^2}\abs{X^{(j)}_{\gamma,1}(t)}}^2.
\end{split}
\end{align}
The last bound implies \eqref{et2}.

\eqref{et1} and \eqref{et2} let us conclude.
\end{proof}

%\partialinput{10}{500}{nah05}

\subsubsection*{Step 2}

In this step, we study the convergence in distribution of $\gamma^{\frac{\theta}{2}-\frac{1}{2}}\ul{\tilde{X}}^{\pare{k}}_\gamma\pare{t_\theta}$.
Actually, the proof of this convergence does not make use of the fact that $\theta<1$, so it works for every positive $\theta$.
By Duhamel's formula \eqref{duha02},
this is equivalent to studying the limit of
\begin{align}\label{27}
\gamma^{\frac{\theta}{2}-\frac{1}{2}}e^{t_\theta A^{\pare{k}} }\ul X_\gamma^{(k)}(0)+\gamma^{\frac{\theta}{2}-\frac{1}{2}}\int_0^{t_\theta} e^{(t_\theta-s)A^{\pare{k} }}\ul{M}^{(k)}_\gamma(ds).
\end{align}

For $k\neq \pm 1$, using propositions \ref{az} and \ref{PPP}, and the fact that the real parts of the eigenvalues are negative, it is easy to see that
\begin{align}
\norm{\gamma^{\frac{\theta}{2}-\frac{1}{2}}e^{t_\theta A^{\pare{k}} }\ul X_\gamma^{(k)}(0)+\gamma^{\frac{\theta}{2}-\frac{1}{2}}\int_0^{t_\theta} e^{(t_\theta-s)A^{\pare{k} }}\ul{M}^{(k)}_\gamma(ds)}\conv{\gamma\to 0}0
\end{align}
in $\BB P$-probability; 
this implies the convergence of $\gamma^{\frac{\theta}{2}-\frac{1}{2}}\ul X^{\pare{k}}_\gamma\pare{t_\theta}$  to the delta measure concentrated in $\pare{
\begin{array}{c}
0
\\[5pt]
0
\end{array}
}$.

The case $k=\pm 1$ is more delicate.
We only do the case $k=1$ as the case $k=-1$ is analogous.
During this proof, we omit writing the superscript $^{\pare{1}}$ when it does not create confusions.
Using decomposition \eqref{decomposition}, we can write
\begin{align}
e^{t_\theta A}\ul X_\gamma\pare{0}
=e^{t_\theta\mu}\proj_1\pare{\ul X_\gamma\pare{0}}
+e^{t_\theta\mu_2}\proj_2\pare{\ul X_\gamma\pare{0}}.
\end{align}
The fact that $\mu_2$ is negative let us proceed similarly to the previous paragraph and conclude that
$e^{t_\theta\mu_2}\gamma^{\frac{\theta}{2}-\frac{1}{2}} \proj_2\pare{\ul X_\gamma\pare{0}}$ vanishes in the limit.
Similarly, also vanishes the projection $\proj_2$ of the second addend of \eqref{27}.
Using identity $\gamma^{\frac{\theta}{2}-\frac{1}{2}} e^{t_\theta\mu}=\gamma^{-\frac{1}{2}}$, the problem is reduced to proving that
\begin{align}
\gamma^{-\frac{1}{2}}\ul X_{\gamma}\pare{0}+\gamma^{-\frac{1}{2}}\int_{0}^{t_\theta}e^{-s\mu_1}\ul M_{\gamma}\pare{ds}
\end{align}
converges in distribution to $\pare{
\begin{array}{c}
U_1^{\R}+\mathtt{i}U_1^{\I} \\[5pt]
U_2^{\R}+\mathtt{i}U_2^{\I} 
\end{array}
}+\pare{
\begin{array}{c}
V_1^{\R}+\mathtt{i}V_1^{\I} \\[5pt]
V_2^{\R}+\mathtt{i}V_2^{\I} 
\end{array}
}$.
We can conclude if we prove that \textit{(a)}
$\gamma^{-\frac{1}{2}}\ul X_{\gamma}\pare{0}$ converges to $\pare{
\begin{array}{c}
U_1^{\R}+\mathtt{i}U_1^{\I} \\[5pt]
U_2^{\R}+\mathtt{i}U_2^{\I} 
\end{array}
}$,
\textit{(b)} $\gamma^{-\frac{1}{2}}\int_{0}^{t_\theta}e^{-s\mu}\ul M_{\gamma}\pare{ds}$  converges to $\pare{
\begin{array}{c}
V_1^{\R}+\mathtt{i}V_1^{\I} \\[5pt]
V_2^{\R}+\mathtt{i}V_2^{\I} 
\end{array}
}$,
and \textit{(c)} 
$\gamma^{-\frac{1}{2}}\ul X_{\gamma}\pare{0}$ and
$\gamma^{-\frac{1}{2}}\int_{0}^{t_\theta}e^{-s\mu}d\ul M_{\gamma}\pare{s}$ are asymptotically independent.

\begin{itemize}
\item[\textit{(a)}]
As $\sigma_{\gamma,1}\pare{0}$ and $\sigma_{\gamma,2}\pare{0}$ are independent for every $\gamma$,
we can conclude if we prove that,
for $i\in\llav{1,2}$, $\gamma^{-\frac{1}{2}}X_{\gamma,i}\pare{0}$ weakly converges to $U_i^{\mathfrak R}+\mathtt{i}U_i^{\mathfrak I}$.
Let $F_{\mathfrak R}\pare{r}\defi \cos\pare{2\pi r}$ and $F_{\mathfrak I}\pare{r}\defi \sin\pare{2\pi r}$ respectively be the real and imaginary parts of $F^{\pare{1}}$.
We prove it only for $i=1$ as the case $i=2$ is analogous.
This goal is equivalent to the convergence of  the random $\BB R^2$-element $\pare{\pic{\sigma_{\gamma,1},F_{\mathfrak R}},\pic{\sigma_{\gamma,1},F_{\mathfrak I}}}$
to $\pare{U_1^{\mathfrak R},U_1^{\mathfrak I}}$.
As the weak convergence of random $\BB R^2$-elements is characterized by the weak convergence of linear combination of the coordinates, we have to prove that,
 for every $a,b\in\BB R$,
the sequence $a\pic{\sigma_{\gamma,1},F_{\mathfrak R}}+b\pic{\sigma_{\gamma,1},F_{\mathfrak I}}=\pic{\sigma_{\gamma,1},aF_{\mathfrak R}+bF_{\mathfrak I}}$ converges weakly to $aU_1^{\mathfrak R}+bU_1^{\mathfrak I}\sim \CAL N\pare{0,a^2\pi+b^2\pi}$.
This follows from proposition \ref{himan}
and the fact that $\int_0^1\pare{aF_{\mathfrak R}+bF_{\mathfrak I}}^2=a^2\pi+b^2\pi$.

\item[\textit{(b)}] 
We give an explicit construction of the probability space $\pare{\Omega,\CAL F,\BB P}$.
It will we the product of two probability spaces $\pare{\Omega_1,\CAL F_1,\BB P_1}$ and $\pare{\Omega_2,\CAL F_2,\BB P_2}$,
the first one for the initial distribution and the second one for the Markovian evolution.
For the first one, we choose $\Omega_1\defi\llav{-1,1}^{\BB N}\times\llav{-1,1}^{\BB N}$ with the product $\sigma$-algebra and the Bernoulli probability with parameter $p=\frac{1}{2}$.
The second one can be any one for which the trajectories are \textit{c\`adl\`ag} for every $\omega_2\in\Omega_2$ (not only almost every):
for instance, we can chose $\Omega_2$ to be the set of subsets of $\BB R^2$ without limit points, $\BB P_2$ to be the Poisson probability with intensity $1$ as in \cite{MR96},
and to construct the Markov process as in \cite{Bre99}.
Fix $\omega_1$  such that $\pic{\sigma_{\gamma,i}\corch{\omega_1},G}\conv{\gamma\to 0}0$ (this occurs with $\BB P_1$-probability $1$).
For simplicity of notation, from now on, the dependence on $\omega_1$ will be ommited.
Decomposing the stochastic integral into the two coordinates and into their real and immaginary parts,
we can conclude if we prove the weak convergence of the $\BB R^4$-random vector $\pare{\mathcal{Z}_{\gamma, 1}^{\mathfrak R}, \CAL Z_{\gamma, 1}^{\mathfrak I} , \CAL Z_{\gamma, 2}^{\mathfrak R}, \CAL Z_{\gamma, 2}^{\mathfrak I}}$ to $\pare{{V}_{ 1}^{\mathfrak R},  V_{1}^{\mathfrak I} ,  V_{2}^{\mathfrak R}, V_{ 2}^{ I}}$,
where 
\begin{align}
\mathcal{Z}_{\gamma,i}^{\#}\defi\gamma^{-\frac{1}{2}}\int_0^{t_c}e^{-s\mu}dM_{\gamma, i}^\#\pare{s}
\end{align}
for all $i\in\llav{1,2}$ and $\#\in \llav{\mathfrak{R}, \mathfrak{I}}$.
As before, we need to prove that, for all $a_1^\mathfrak{R}, a_1^\mathfrak{I}, a_2^\mathfrak{R}, a_2^\mathfrak{I}\in\BB R$, 
\begin{align}\label{ropo}
\sum_{\#\in\llav{\mathfrak R, \mathfrak I}} \, \sum_{i\in\llav{1,2}}a_i^{\#}\mathcal{Z}_{\gamma, i}^{ \#} \
\xrightarrow[\gamma\to 0]{D} \
\sum_{\#\in\llav{\mathfrak R, \mathfrak I}} \,
\sum_{i\in\llav{1,2}}a_i^{\#}{V}_i^{\#}\sim N\pare{0, \sum_{\#\in\llav{\mathfrak R, \mathfrak I}}\, \sum_{i\in\llav{1,2}}\pare{a_i^{\#}}^2\Var\pare{V_i^{\#}}}.
\end{align}
Observe that the left hand side of \eqref{ropo} can be written as $\gamma^{-\frac{1}{2}}\int_0^{t_c}e^{-s\mu}\tilde dM_\gamma\pare{s}$,
where $\tilde M_\gamma\pare{t}$ is the martingale with respect to the filtration of the process defined as
\begin{align}
\begin{aligned}
\tilde M_\gamma\pare{t}\defi&\pic{\sigma_{\gamma, 1}\pare{t}, a_1^{\mathfrak R}F_{\mathfrak{R}}+a_1^{\mathfrak I}F_{\mathfrak{I}}}+\pic{\sigma_{\gamma, 2}\pare{t}, a_2^{\mathfrak R}F_{\mathfrak{R}}+a_2^{\mathfrak I}F_{\mathfrak{I}}}\\[5pt]
&-\pic{\sigma_{\gamma, 1}\pare{0}, a_1^{\mathfrak R}F_{\mathfrak{R}}+a_1^{\mathfrak I}F_{\mathfrak{I}}}-\pic{\sigma_{\gamma, 2}\pare{0}, a_2^{\mathfrak R}F_{\mathfrak{R}}+a_2^{\mathfrak I}F_{\mathfrak{I}}}\\[5pt]
&-\int_0^{t_c}L_\gamma\pare{\pic{\sigma_{\gamma, 1}\pare{s}, a_1^{\mathfrak R}F_{\mathfrak{R}}+a_1^{\mathfrak I}F_{\mathfrak{I}}}+\pic{\sigma_{\gamma, 2}\pare{s}, a_2^{\mathfrak R}F_{\mathfrak{R}}+a_2^{\mathfrak I}F_{\mathfrak{I}}}}ds.
\end{aligned}
\end{align}
Call $I_\gamma\pare{t}\defi\gamma^{-\frac{1}{2}}\int_0^{t}e^{-s\mu}d\tilde M_\gamma\pare{s}$. \eqref{ropo} follows if we prove that 
\begin{itemize}
\item[1.] $\abs{I_\gamma\pare{t_c}-I_\gamma\pare{\infty}}\xrightarrow[\gamma\to 0]{\BB P_2}0$ and
\item[2.] $I_\gamma\pare{\infty}\xrightarrow[\gamma\to 0]{D}N\pare{0,\sum_{\#\in\llav{\mathfrak R, \mathfrak I}}\sum_{i\in\llav{1,2}}\pare{a_i^{\#}}^2V\pare{V_i^{\#}}}$.
\end{itemize}

We first prove item $1$.
From now on, $\pic{M}\pare{\cdot}$ will denote the compensator of the process $M\pare{\cdot}$  (see \cite{HHK06} for details).
Fix $\varepsilon>0$, and apply Chebyshev inequality and It\^o isometry to get
\begin{align}\label{pl}
\BB P_2\pare{\abs{I_\gamma\pare{t_c}-I_\gamma\pare{\infty}}>\varepsilon}\leq\varepsilon^{-2}\BB E_{\BB P_2}\pare{\int_{t_c}^{\infty}e^{-2\mu_1 s}\pic{\tilde M_\gamma}\pare{ds}}.
\end{align}
Call $G_\gamma\pare{t}\defi \pic{\sigma_{\gamma, 1}\pare{t}, a_1^{\mathfrak R}F_{\mathfrak{R}}+a_1^{\mathfrak I}F_{\mathfrak{I}}}+\pic{\sigma_{\gamma, 2}\pare{t}, a_2^{\mathfrak R}F_{\mathfrak{R}}+a_2^{\mathfrak I}F_{\mathfrak{I}}}$.
By Lemma 5.1 in the Appendix of \cite{KL99}, we know that
\begin{align}\label{hj}
\begin{aligned}
\pic{\tilde M_\gamma}\pare{t}=\int_0^t L_\gamma G_\gamma\pare{s}^2-2G_\gamma\pare{s}L_\gamma G_\gamma\pare{s}ds.
\end{aligned}
\end{align}
Easy computations show that there exists a constant $C_1$ such that, for all $t\geq 0$,  $L_\gamma G_\gamma\pare{t}^2-2G_\gamma\pare{t}L_\gamma G_\gamma\pare{t}\leq C_1\gamma $. This last estimation, \eqref{pl} and \eqref{hj} allow us to conclude item $1$.

We now prove item $2$.
The following Lemma holds.
\begin{lemma}\label{lkj88} For all $t\geq0$,
\begin{align}
I_\gamma\pare{t}\xrightarrow[\gamma\to 0]{D}I_0\pare{{t}},
\end{align}
where $I_0\pare{t}\sim N\pare{0,H\pare{t}}$, with $H\pare{t}=\int_0^te^{-2\mu s}h\pare{s}\ds$ and
\begin{align}
\begin{aligned}
h(t)=&\corch{\pare{a_1^{\mathfrak R}}^2+\pare{a_1^{\mathfrak I}}^2}
\corch{2\pi+{\pi}\tanh\pare{\beta_1\lambda}\pare{\tanh\pare{\beta_1\lambda}-\tanh\pare{\beta_2\lambda}}\pare{1-e^{-2t}}}\\[5pt]
&+\corch{\pare{a_2^{\mathfrak R}}^2+\pare{a_2^{\mathfrak I}}^2}
\corch{2\pi+{\pi}\tanh\pare{\beta_2\lambda}\pare{\tanh\pare{\beta_1\lambda}-\tanh\pare{\beta_2\lambda}}\pare{1-e^{-2t}}}.
\end{aligned}
\end{align}
\end{lemma}
\begin{proof}
To conclude, it is enough to prove that
\begin{align}
\pic{I_\gamma }\pare{t}\xrightarrow[\gamma\to 0]{\BB P_2}H\pare{t}
\end{align}
for all $t\geq 0$ (see Theorem 13 in \cite{Pol84}, for instance).
Observe that 
\begin{align}\label{a14}
&\abs{\pic{I_\gamma}\pare{t}-H\pare{t}}=\abs{\gamma^{-1}\int_0^t e^{-2\mu s}\pic{\tilde M_{\gamma}}\pare{ds}-
\int_0^t e^{-2\mu s} h\pare{s}\ds}.
\end{align}
Use formula \eqref{hj} to get
\begin{align}\label{09}
\begin{aligned}
\gamma^{-1}\pic{\tilde M_{\gamma}}\pare{t}&=\gamma^{-1}\int_0^t \sum_{x\in\Lambda_\gamma}R_1\pare{x, \underline\sigma_\gamma\pare{s}}\pare{\pic{\sigma_{\gamma, 1}^x\pare{s}, a_1^{\mathfrak R}F_{\mathfrak R}}-\pic{\sigma_{\gamma, 1}\pare{s}, a_1^{\mathfrak R}F_{\mathfrak R}}}^2 \\[5pt]
&\blanco{=\gamma^{-1}\int_0^t}+\sum_{x\in\Lambda_\gamma}R_2\pare{x, \underline\sigma_\gamma\pare{s}}\pare{\pic{\sigma_{\gamma, 2}^x\pare{s}, a_2^{\mathfrak R}F_{\mathfrak R}}-\pic{\sigma_{\gamma, 2}\pare{s}, a_2^{\mathfrak R}F_{\mathfrak R}}}^2\ds\\[5pt]
&=\int_0^t 4\gamma\sum_{x\in\Lambda_\gamma}R_1\pare{x, \underline\sigma_\gamma\pare{s}}\pare{a_1^{\mathfrak R}F_{\mathfrak R}\pare{\gamma x}+a_1^{\mathfrak I}F_{\mathfrak I}\pare{\gamma x}}^2 \\[5pt]
&\blanco{=\int_0^t}+4\gamma\sum_{x\in\Lambda_\gamma}R_2\pare{x, \underline\sigma_\gamma\pare{s}}\pare{a_2^{\mathfrak R}F_{\mathfrak R}\pare{\gamma x}+a_2^{\mathfrak I}F_{\mathfrak I}\pare{\gamma x}}^2\ds.
\end{aligned}
\end{align}
Call $h_\gamma\pare{t}$ the integrand appearing in the last expression.
By \eqref{a14}, the proof follows if we prove that
\begin{align}\label{poiuyt}
\sup_{s\in\corch{0,t}}\abs{h_{\gamma}\pare{s}-h\pare{s}}\xrightarrow[\gamma\to 0]{\BB P_2}0
\end{align}
for all $t>0$.
Before proving \eqref{poiuyt}, we prove that
\begin{align}\label{hjh}
\sup_{s\in [0,t]} \, \sup_{x\in \mathbb{T}_\gamma}\abs{\sigma_{\gamma, i}\pare{s}\ast \phi_i\pare{x}}\xrightarrow[\gamma\to 0]{\BB P_2}0
\end{align}
for all $i\in\llav{1,2}$.
Fix $\varepsilon>0$ and $N\in \BB N$ such that $\sum_{|k|> N}e^{-k^2c_1}<\frac{\varepsilon}{2}$.
Then
\begin{equation}
\begin{split}
\sup_{s\in [0,t]}\, \sup_{x\in \mathbb{T}_\gamma}\abs{\sigma_{\gamma, 1}\pare{s}\ast \phi_i\pare{x}}&=\sup_{s\in [0,t]} \, \sup_{x\in \mathbb{T}_\gamma}\abs{\sum_{k\in\mathbb{Z}}\hat{\phi}_{c_1}^{(k)}\hat\sigma_{\gamma,1}^{(k)}(s) e^{i2\pi kx}}\\[5pt]
&\leq\sum_{|k|\leq N}e^{-k^2c_1}\sup_{s\in[0,t]}\abs{\pic{\sigma_{\gamma, 1}(s), F^{(k)}}}+\frac{\varepsilon}{2}.
\end{split}
\end{equation}
By Corollary \ref{arjona}, we know that, for every $ k\in \BB Z$,
\begin{align}
\sup_{s\in[0,t]}\abs{\pic{\sigma_{\gamma, 1}(s), F^{(k)}}}\xrightarrow[\gamma\to 0]{\BB{P}_2}0,
\end{align}
and hence \eqref{hjh} holds for $i=1$ follows.
The proof for $i=2$ is similiar.
We proceed with the proof of \eqref{poiuyt}.
For $i\in\llav{1,2}$, call $G_i\pare{r}=a_i^{\mathfrak R}F_{\mathfrak R}\pare{r}+a_i^{\mathfrak I}F_{\mathfrak I}\pare{r}$.
By writing the coefficient $1$ in front of $R_1\pare{x, \underline \sigma_\gamma}$ as $\frac{1+\sigma_{\gamma, 1}\pare{x}}{2}+\frac{1-\sigma_{\gamma,1}\pare{x}}{2}$ and the coefficient $1$ in front of $R_2\pare{x, \underline \sigma_\gamma}$ as $\frac{1+\sigma_{\gamma, 2}\pare{x}}{2}+\frac{1-\sigma_{\gamma,2}\pare{x}}{2}$,
we obtain that
\begin{align}\label{hjk}
h_\gamma\pare{s}=&2\gamma\sum_{x\in \Lambda_\gamma}\pare{G_1\pare{\gamma x}}^2-2\gamma\sum_{x\in \Lambda_\gamma}\sigma_{\gamma, 1}\pare{x,s}\tanh\pare{\beta_1\pare{\sigma_{\gamma, 1}\ast\phi_1}\pare{x}+\beta_1\lambda\sigma_{\gamma, 2}\pare{x}}\pare{G_1\pare{\gamma x}}^2\\[5pt]
&+2\gamma\sum_{x\in \Lambda_\gamma}\pare{G_2\pare{\gamma x}}^2-2\gamma\sum_{x\in \Lambda_\gamma}\sigma_{\gamma, 2}\pare{x,s}\tanh\pare{\beta_2\pare{\sigma_{\gamma, 2}\ast\phi_2}\pare{x}-\beta_2\lambda\sigma_{\gamma, 1}\pare{x}}\pare{G_2\pare{\gamma x}}^2.
\end{align}
Proceeding as before, we get
\begin{align}\label{h5}
\begin{aligned}
h_\gamma\pare{s}
=&2\gamma\sum_{x\in \Lambda_\gamma}\pare{G_1\pare{\gamma x}}^2\\[5pt]
&-\gamma\sum_{x\in{\Lambda}_\gamma}\sigma_{\gamma, 1}(x,s)\corch{\tanh\pare{\beta_1\pare{\sigma_{\gamma, 1}\ast\phi_1}\pare{x}+\beta_1\lambda}+\tanh\pare{\beta_1\pare{\sigma_{\gamma, 1}\ast\phi_1}\pare{x}-\beta_1\lambda}}\pare{G_1\pare{\gamma x}}^2\\[5pt]
&-\gamma\sum_{x\in  \Lambda_\gamma}\eta_{\gamma}(x,s)\corch{\tanh\pare{\beta_1\pare{\sigma_{\gamma, 1}\ast\phi_1}\pare{x}+\beta_1\lambda}-\tanh\pare{\beta_1\pare{\sigma_{\gamma, 1}\ast\phi_1}\pare{x}-\beta_1\lambda}}\pare{G_1\pare{\gamma x}}^2\\[5pt]
&+2\gamma\sum_{x\in \Lambda_\gamma}\pare{G_2\pare{\gamma x}}^2\\[5pt]
&-\gamma\sum_{x\in{\Lambda}_\gamma}\sigma_{\gamma, 2}(x,s)\corch{\tanh\pare{\beta_2\pare{\sigma_{\gamma, 2}\ast\phi_2}\pare{x}+\beta_2\lambda}+\tanh\pare{\beta_2\pare{\sigma_{\gamma, 2}\ast\phi_2}\pare{x}-\beta_2\lambda}}\pare{G_1\pare{\gamma x}}^2\\[5pt]
&+\gamma\sum_{x\in  \Lambda_\gamma}\eta_{\gamma}(x,s)\corch{\tanh\pare{\beta_2\pare{\sigma_{\gamma, 2}\ast\phi_2}\pare{x}+\beta_2\lambda}-\tanh\pare{\beta_2\pare{\sigma_{\gamma, 2}\ast\phi_1}\pare{x}-\beta_2\lambda}}\pare{G_2\pare{\gamma x}}^2.
\end{aligned}
\end{align}
Observe that
\begin{align}\label{h6}
2\gamma\sum_{x\in \Lambda_\gamma}\pare{G_1\pare{\gamma x}}^2+2\gamma\sum_{x\in \Lambda_\gamma}\pare{G_2\pare{\gamma x}}^2\xrightarrow[\gamma\to 0]{}2\int_{\BB T}G_1\pare{r}^2\dr+2\int_{\BB T}G_2\pare{r}^2\dr.
\end{align}
From the Lipschitzianity of the function $\tanh(\cdot)$, we know that there exists a constant $C_2$ such that
\begin{align}\label{a12}
\begin{aligned}
&\sup_{s\in[0,t]}\abs{\gamma\sum_{x\in\Lambda_\gamma}\sigma_{\gamma, 1}(x, s)\pare{\tanh\pare{\beta_1\pare{\sigma_{\gamma, 1}\ast\phi_1}\pare{x}+\beta_1\lambda}+\tanh\pare{\beta_1\pare{\sigma_{\gamma, 1}\ast\phi_1}\pare{x}-\beta_1\lambda}}\pare{G_1\pare{\gamma x}}^2}\\[5pt]
&\quad\leq C_2\sup_{s\in[0,t]}\,\sup_{x\in\Lambda_\gamma}\,\abs{\sigma_{\gamma, 1}\pare{s}\ast \phi_1\pare{x}}\int_{\BB T}G_1\pare{r}^2\dr\xrightarrow[\gamma\to 0]{\BB P_2}0.
\end{aligned}
\end{align}
Analogously, we get
\begin{align}\label{a154}
\begin{aligned}
&\sup_{s\in[0,t]}\abs{\gamma\sum_{x\in\Lambda_\gamma}\sigma_{\gamma, 2}(x, s)\pare{\tanh\pare{\beta_2\pare{\sigma_{\gamma, 1}\ast\phi_2}\pare{x}+\beta_2\lambda}+\tanh\pare{\beta_2\pare{\sigma_{\gamma, 2}\ast\phi_2}\pare{x}-\beta_2\lambda}}\pare{G_2\pare{\gamma x}}^2}\\[5pt]
&\quad\xrightarrow[\gamma\to 0]{\BB P_2}0.
\end{aligned}
\end{align}
To conclude, we just need to compute the limit in probability of the terms involving the correlation field in the right hand side of \eqref{h5}. Observe that
\begin{align}\label{019}
\begin{aligned}
&\sup_{s\in [0, t]}\Big|\gamma\sum_{x\in  \Lambda_\gamma}\eta_{\gamma}(x,s)\pare{\tanh\pare{\beta_1\pare{\sigma_{\gamma, 1}\ast\phi_1}\pare{x}+\beta_1\lambda}-\tanh\pare{\beta_1\pare{\sigma_{\gamma, 1}\ast\phi_1}\pare{x}-\beta_1\lambda}}\pare{G_1\pare{\gamma x}}^2\\[5pt]
&\blanco{\sup_{s\in [0, t]}\bigg|}-\tanh\pare{\beta_1\lambda}\pare{\tanh\pare{\beta_1\lambda}-\tanh\pare{\beta_2\lambda}}\pare{1-e^{-2s}}\int_{\BB T}G_1\pare{r}^2\dr\bigg|\\[5pt]
&\quad\leq 2C_2\sup_{s\in [0, t]}\sup_{x\in\Lambda_\gamma}\abs{\sigma_{\gamma, 1}\pare{s}\ast \phi_1\pare{x}}\\[5pt]
&\quad\blanco{\le}+\sup_{s\in [0, t]}\bigg|2\gamma\sum_{x\in\Lambda_\gamma}\eta_\gamma\pare{x,s}\tanh\pare{\beta_1\lambda}
\\[5pt]
&\quad\blanco{\le+\sup_{s\in [0, t]}\bigg|}
-\tanh\pare{\beta_1\lambda}\corch{\tanh\pare{\beta_1\lambda}
-\tanh\pare{\beta_2\lambda}}\pare{1-e^{-2s}}\int_{\BB T}G_1\pare{r}^2\dr\bigg|
\\[5pt]
&\quad\blanco{\le}\xrightarrow[\gamma\to 0]{}0.
\end{aligned}
\end{align}
The last convergence is a consequence of \eqref{hjh} and Corollary \ref{arjona}. Analogously, we can prove that
\begin{align}\label{020}
\begin{aligned}
&\sup_{s\in [0, t]}\Big|\gamma\sum_{x\in  \Lambda_\gamma}\eta_{\gamma}(x,s)\pare{\tanh\pare{\beta_2\pare{\sigma_{\gamma, 2}\ast\phi_2}\pare{x}+\beta_2\lambda}-\tanh\pare{\beta_2\pare{\sigma_{\gamma, 2}\ast\phi_2}\pare{x}-\beta_2\lambda}}\pare{G_2\pare{\gamma x}}^2\\[5pt]
&\blanco{\sup_{s\in [0, t]}\Big|}-\tanh\pare{\beta_2\lambda}\pare{\tanh\pare{\beta_1\lambda}-\tanh\pare{\beta_2\lambda}}\pare{1-e^{-2s}}\int_{\BB T}G_2\pare{r}^2\dr\Big|
\\[5pt]
&\quad\xrightarrow[\gamma\to 0]{}0.
\end{aligned}
\end{align}
For $i\in\llav{1,2}$, $\int_{\BB T}G_i\pare{r}^2\dr=\pi\corch{\pare{a_i^{\mathfrak R}}^2+\pare{a_i^{\mathfrak I}}^2}$ as $\int_{\BB T}F_{\mathfrak R}\pare{r}F_{\mathfrak I}\pare{r}\dr=0$. Then \eqref{h5}-\eqref{020} allow us to conclude \eqref{poiuyt}.
\end{proof}
Lemma \ref{lkj88} implies item 2.
Observe that $I_0\pare{t}$ converges in distribution to 
$I_0\pare{\infty}\sim  N\Bigg(0, \sum_{\#\in\llav{\mathfrak R, \mathfrak I}}\sum_{i\in\llav{1,2}}\pare{a_i^{\#}}^2\Var\pare{V_i^{\#}}\Bigg)$ as 
$\lim_{t\to \infty}H\pare{t}= \sum_{\#\in\llav{\mathfrak R, \mathfrak I}}\sum_{i\in\llav{1,2}}\pare{a_i^{\#}}^2\Var\pare{V_i^{\#}}$.
Fix $a\in \BB R$ and $\varepsilon>0$.
The proof of item 2 follows if we show that there exists a $\bar\gamma$ such that  
\begin{align}\label{258}
\abs{\BB P_2\pare{I_\gamma\pare{\infty}\leq a}- \BB P_2\pare{I_0\pare{\infty}\leq a}}<\varepsilon
\end{align}
for every $\gamma<\bar \gamma$.
Fix $T>0$ such that 
\begin{align}
&\abs{\BB P_2\pare{I_0\pare{T}\leq a}- \BB P_2\pare{I_0\pare{\infty}\leq a}}<\frac{\varepsilon}{3},
\\[5pt]
&\BB P_2\pare{a-e^{-\frac{\mu}{2}T}\leq I_0\pare{T}\leq a}<\frac{\varepsilon}{24},  
\\[5pt]
&\textnormal{and } C_1e^{-\mu_1 T}<\frac{\varepsilon}{12},
\end{align}
where $C_1$ is a constant which will be specified later (the last choice of $T$ can be done because the limiting distribution $I_0\pare{\infty}$ is continuous).
Then
\begin{align}\label{259}
\begin{aligned}
\abs{\BB P_2\pare{I_\gamma\pare{\infty}\leq a}- \BB P_2\pare{I_0\pare{\infty}\leq a}}\leq&\abs{\BB P_2\pare{I_\gamma\pare{\infty}\leq a}- \BB P_2\pare{I_\gamma\pare{T}\leq a}}\\[5pt]
&+\abs{\BB P_2\pare{I_\gamma\pare{T}\leq a}- \BB P_2\pare{I_0\pare{T}\leq a}}\\[5pt]
&+\abs{\BB P_2\pare{I_0\pare{T}\leq a}- \BB P_2\pare{I_0\pare{\infty}\leq a}}.
\end{aligned}
\end{align}
It is easily seen that there exists $\gamma_1$ such that 
\begin{align}
\begin{aligned}
\abs{\BB P_2\pare{I_\gamma\pare{T}\leq a}-\BB  P_2\pare{I_0\pare{T}\leq a}}+\abs{\BB P_2\pare{I_0\pare{T}\leq a}-\BB P_2\pare{I_0\pare{\infty}\leq a}}<\frac{2}{3}\varepsilon
\end{aligned}
\end{align}
for all $\gamma<\gamma_1$.
By \eqref{259}, it is enough to prove that there exists $\gamma_2$ such that  
\begin{align}\label{300}
\begin{aligned}
\abs{\BB P_2\pare{I_\gamma\pare{\infty}\leq a}- \BB P_2\pare{I_\gamma\pare{T}\leq a}}<\frac{\varepsilon}{3}
\end{aligned}
\end{align}
for all $\gamma<\gamma_2$.
Call $\CAL A_{\gamma, T}:=\llav{\omega_2:\abs{I_\gamma\pare{\infty}- I_\gamma\pare{T}}>e^{-\frac{\mu_1}{2} T}}$.
By Chebysev inequality, Ito's isometry, and \eqref{09}, it is possible to prove that there exists a constant $C_1$ such that
\begin{align}\label{400}
\BB P_2\pare{\CAL A_{\gamma, T}}\leq C_1 e^{-\mu T}
\end{align}
for all $\gamma>0$.
By \eqref{400}, we get
\begin{align}\label{ghg}
\begin{aligned}
&\abs{\BB P_2\pare{I_\gamma\pare{\infty}\leq a}-\BB P_2\pare{I_\gamma\pare{T}\leq a}}
\\[5pt]
&\quad\leq 2C_1 e^{-\mu T}+\BB P_2\pare{I_\gamma\pare{T}
\leq a, I_\gamma\pare{\infty}>a, \CAL A_{\gamma, T}^c}\\[5pt]
&\quad\blanco{\le}+P_2\pare{I_\gamma\pare{T}> a, I_\gamma\pare{\infty}\leq a, \CAL A_{\gamma, T}^c}\\[5pt]
&\quad\leq\frac{\varepsilon}{6}+2\BB P_2\pare{a-e^{-\frac{\mu}{2}T}\leq I_\gamma\pare{T}\leq a}\\[5pt]
&\quad\leq\frac{\varepsilon}{6}+2\abs{\BB P_2\pare{a-e^{-\frac{\mu}{2}T}\leq I_\gamma\pare{T}\leq a}-\BB P_2\pare{a-e^{-\frac{\mu}{2}T}\leq I_0\pare{T}\leq a}}\\[5pt]
&\quad\blanco{\le}+2\BB P_2\pare{a-e^{-\frac{\mu}{2}T}\leq I_0\pare{T}\leq a}.
\end{aligned}
\end{align}
Our choice of $\gamma_2$ let us bound last expression by $\frac{\varepsilon}{3}$, obtaining \eqref{300}.

\item[\textit{(c)}] 
We use the construction $\Omega=\Omega_1\times \Omega_2$ defined in item \textit{(a)}.
We prove that $\ul{U}$ and $\ul{V}$ are independent.
Let
$
\ul U_\gamma$ and $\ul V_\gamma$ respectively be $\gamma^{-\frac{1}{2}}\ul X_\gamma\pare{0}$ and 
$\gamma^{-\frac{1}{2}}\int_0^{t_\theta}e^{-s\mu}d\ul M_\gamma\pare{s}$ thought as $\BB R^4$-random vectors. 
Let $A,B\subset \BB R^4$ respectively of the form $\prod_{i=1}^4\left(-\infty,a_i\right]$ and $\prod_{i=1}^4\left(-\infty,b_i\right]$.
We are done if we prove that
\begin{align}
\BB P\corch{\ul{U}_\gamma\in A,\ul{V}_\gamma\in B}\conv{\gamma\to 0}
\BB P\corch{\ul{U}\in A}\BB P\corch{\ul{V}\in B},
\end{align}
where again we think $\ul U$ and $\ul V$ as  $\BB R^4$-random vectors.
We have
\begin{align}
\begin{aligned}\nonumber
\BB P\corch{\ul{U}_\gamma\in A,\ul{V}_\gamma\in B}&
=\int \BB P_1\pare{d\omega_1}\int \BB P_2\pare{d\omega_2}
\textbf{1}\llav{\ul U_\gamma\corch{\omega_1}\in A}\textbf{1}\llav{\ul V_\gamma\corch{\omega_1,\omega_2}\in B}
\\[5pt]
&=\int \BB P_1\pare{d\omega_1}f_{A,\gamma}\corch{\omega_1}g_{B,\gamma}\corch{\omega_1}
\end{aligned}
\end{align}
with $f_{A,\gamma}\corch{\omega_1}\defi \textbf{1}\llav{\ul U_\gamma\corch{\omega_1}\in A}$ and $g_{B,\gamma}\corch{\omega_1}\defi \int \BB P_2\pare{d\omega_2}\textbf{1}\llav{\ul V_\gamma\corch{\omega_1,\omega_2}\in B}$.
Convergence $g_{B,\gamma}\corch{\omega_1}\conv{\gamma\to 0}g_B$ holds for $\BB P_1$-almost every $\omega_1$, where $g_B$ does not depend on $\omega_1$.
As $f_{B,\gamma}$ is uniformly bounded (by $1$),
\begin{align}
\abs{
\int \BB P_1\pare{d\omega_1}f_{A,\gamma}\corch{\omega_1}g_{B,\gamma}\corch{\omega_1}
-g_B\int \BB P_1\pare{d\omega_1}f_{A,\gamma}\corch{\omega_1}}
\le
\int \BB P_1\pare{d\omega_1}\abs{g_{B,\gamma}\corch{\omega_1}-g_B};
\end{align}
the later integral vanishes because of the dominated convergence theorem.
Using that $\int \BB P_1\pare{d\omega_1}f_{A,\gamma}\corch{\omega_1}\conv{\gamma}\BB P\corch{\ul U\in A}$ thanks to the convergence in distribution of $\ul U_\gamma$ to $\ul U$, we get 
\begin{align}
\int \BB P_1\pare{d\omega_1}f_{A,\gamma}\corch{\omega_1}g_{B,\gamma}\corch{\omega_1}\conv{\gamma\to 0}g_B \BB P\corch{\ul U\in A}.
\end{align}

\end{itemize}

\subsection{Proof of Theorem \eqref{2}}

We prove that, for $k=\pm 1$,
\begin{align}\label{fine}
\lim_{\delta\to 0}\liminf_{\gamma\to 0}
\BB P\pare{\norm{\ul X_\gamma^{(k)}\pare{t_c-T_\delta}}>\delta}=1,
\end{align}
where $t_c=\frac{1}{2\mu}\log\gamma^{-1}$ and $T_\delta=\frac{1}{2\mu}\log\delta^{-1}$.
We only prove it for $k=1$.
For this purpose, we define three events $\SCR{A}_\gamma$, $\SCR{B}_{\delta,\gamma}$ and $\SCR C_{\delta,\gamma}$ that satisfy $\SCR{A}_\gamma\cap \SCR{B}_{\delta,\gamma}\cap \SCR C_{\delta,\gamma}\subset \pare{\norm{\ul X_\gamma^{(k)}\pare{t_c-T_\delta}}>\delta}$ for $\delta$ and $\gamma$ small enough, and 
\begin{align}\label{tero0}
\lim_{\delta\to 0}\liminf_{\gamma\to 0}
\BB P\pare{\SCR{A}_\gamma\cap \SCR{B}_{\delta,\gamma}\cap \SCR C_{\delta,\gamma}}=1.
\end{align}

The event $\SCR A_\gamma$ is defined as
\begin{align}
\begin{aligned}
\mathscr A_{\gamma}
=\corch{\sup_{t\in [0,\infty)}\norm{\int_0^t e^{\pare{t-s} A^{\pare{k}}}\ul M_\gamma^{\pare{k}}\pare{ds}}\leq k^2 \gamma^{\frac{1}{4}}\quad\forall k\in\BB Z\setminus\llav{\pm 1}, \norm{\ul{X}_\gamma^{\pare{k}}\pare{0}}\leq k^2 \gamma^{\frac{1}{4}}\quad\forall k\in\BB Z}.
\end{aligned}
\end{align}
Lemma \ref{PPP} guarantees
\begin{align}\label{tero1}
\BB P\pare{\SCR A_\gamma}\conv{\gamma\to 0}1.
\end{align}
The event $\SCR C_{\delta,\gamma}$ is defined as $\corch{\norm{\ul{\tilde X}^{\pare{1}}_\gamma\pare{t_c-T_\delta}}>\delta^{\frac{3}{4}}}$. 
It satisfies
\begin{align}\label{tero2}
\lim_{\delta\to 0}\liminf_{\gamma\to 0}\BB P\pare{\SCR C_{\delta,\gamma}}=1
\end{align}
because $\tilde{\ul X}_\gamma\pare{t_c}$ has the same limiting distribution than $\gamma^{\frac{\theta}{2}-\frac{1}{2}}\tilde{\ul X}_\gamma\pare{t_\theta}$ 
%\blue{what about stating step 2 with general $\theta$, not necessarily in $\pare{0,\frac{1}{2}}$?}
and $\tilde{\ul X}_\gamma^{\pare{1}}\pare{t_c-T_\delta}$ the same limiting distribution than $\delta^{\frac{1}{2}}\tilde{\ul X}_\gamma^{\pare{1}}\pare{t_c}$.
To define $\SCR B_{\delta,\gamma}$, we need to introduce some random-variables.
For $k\in\BB Z$, let
\begin{align}
&a_\gamma^{(k)}\defi \gamma^{-\frac{1}{2}}\int_0^\infty e^{-sA^{\pare{k}}}\ul M_\gamma^{\pare{k}}\mm\pare{ds}
\\[5pt]
&b_\gamma^{\pare{k}}\defi\sup_{t\ge 0}\norm{\gamma^{-\frac{1}{2}}\int_t^\infty e^{-sA^{\pare{k}}}\ul M_\gamma^{\pare{k}}\mm\pare{ds}}
\\[5pt]
&r_\gamma^{(k)}\defi \gamma^{-\frac{1}{2}}\ul X_\gamma^{\pare{k}}\mm\pare{0}.
\end{align}
Then $\SCR B_{\delta,\gamma}$ is defined as $\SCR B_{\delta,\gamma}\defi\corch{\delta^{\frac{1}{4}}<R_\gamma<\delta^{-\frac{1}{16}}}$
with
\begin{align}
R_\gamma\defi \max\llav{\norm{a_\gamma^{\pare{-1}}},\norm{a_\gamma^{\pare{1}}}}+
\max\llav{b_\gamma^{(-1)},b_\gamma^{(1)}}+\max\llav{\norm{r_\gamma^{(-1)}},\norm{r_\gamma^{(1)}}}.
\end{align}
Once we have \eqref{tero1} and \eqref{tero2}, \eqref{tero0} follows if we prove
\begin{align}\label{house}
\lim_{\delta\to 0}\liminf_{\gamma\to 0}\BB P\pare{\SCR B_{\delta,\gamma}}=1;
\end{align}
we postpone the proof of this fact to the end of the subsection.

From now on, we suppose we are in $\SCR A_\gamma\cap\SCR B_{\delta,\gamma}$.
An intermediate step will be to prove that the stopping time
\begin{align}
t^*\defi \inf\llav{t\ge 0:\sup_{\abs{k}\le \gamma^{-\frac{3}{4}}}\norm{\ul E_\gamma^{\pare{k}}\pare{t}}>DR_\gamma^2\gamma e^{2\mu t}}
\end{align}
is strictly larger than $t_c-T_\delta$ if $D$ is large enough.
We give a name to the expression appearing in the definition of $t^*$:
\begin{align}
g_\gamma\pare{t}\defi DR_\gamma^2\gamma e^{2\mu t}.
\end{align}

First observe that, for $k=\pm 1$, by Duhamel's formula and Lemma \ref{az}, there exists a constant $C_1$ such that
\begin{align}\label{oa1}
\norm{\ul X_\gamma^{\pare{k}}\mm\pare{t}}\le
C_1\corch{R_\gamma e^{t\mu}\gamma^{\frac{1}{2}}+g_\gamma\pare{t}}=
C_1\corch{D^{-\frac{1}{2}}g_\gamma\pare{t}^{\frac{1}{2}}+g_\gamma\pare{t}}
\end{align}
$\forall t\in\left[0,\min\llav{t^*,t_c-T_\delta}\right)$ (here we do not use that we are in the event $\SCR A_\gamma\cap\SCR B_{\delta,\gamma}$).
Also, if $\abs{k}\le \gamma^{-\frac{3}{4}}$ and $k\neq \pm 1$,
\begin{align}\label{oa2}
\norm{\ul X_\gamma^{\pare{k}}\pare{t}}\le C_2\corch{k^2\gamma^{\frac{1}{4}}+g_\gamma\pare{t}}
\end{align}
for every $t\in\left[0,\min\llav{t^*,t_c-T_\delta}\right)$;
here we used Duhamel's formula,
Lemma \ref{az}, and the fact that we are in $\SCR A_\gamma$.
If $C_3$ is the constant of lemma \ref{bound1}, we have
\begin{align}\label{haydn}
\sup_{\abs{k}\le \gamma^{-\frac{3}{4}}}\norm{\ul E_\gamma^{\pare{k}}\pare{t}}
\le
C_3\corch{\gamma^{\frac{1}{4}}+\sum_{j\in\BB Z}e^{-c_1j^2}\norm{\ul X_\gamma^{\pare{k}}\pare{t}}}^2.
\end{align}
For the term inside the square brackets,
divide the sum as $\sum_{j\in\BB Z}=\sum_{j=\pm 1}+\sum_{\substack{\abs{j}\le \gamma^{-\frac{3}{4}}\\[3pt] j\neq \pm 1}}+\sum_{\abs{j}>\gamma^{-\frac{3}{4}}}$,
and use estimate \eqref{oa1} for the first term, estimate \eqref{oa2} for the second one, and estimate $\sum_{\abs{j}>\gamma^{-\frac{3}{4}}}e^{-c_1j^2}\norm{\ul X_\gamma^{\pare{j}}\pare{t}}\le C_4\gamma^{\frac{1}{4}}$ for the third one, to get
\begin{align}
\gamma^{\frac{1}{4}}+\sum_{j\in\BB Z}e^{-c_1j^2}\norm{\ul X_\gamma^{\pare{k}}\pare{t}}
\le C_5\corch{\gamma^{\frac{1}{4}}+R_\gamma e^{t\mu}\gamma^{\frac{1}{2}}+g_\gamma\pare{t}}
=C_5\corch{\gamma^{\frac{1}{4}}+D^{-\frac{1}{2}}g_\gamma\pare{t}^{\frac{1}{2}}+g_\gamma\pare{t}}
\end{align}
for a proper constant $C_5$.
Going back to \eqref{haydn}, we get
\begin{align}
\sup_{\abs{k}\le \gamma^{-\frac{3}{4}}}\norm{\ul E_\gamma^{\pare{k}}\pare{t}}
\le
C_6\corch{\gamma^{\frac{1}{4}}+D^{-\frac{1}{2}}g_\gamma\pare{t}^{\frac{1}{2}}+g_\gamma\pare{t}}^2\fide f_\gamma\pare{t}.
\end{align}
As before, we get $t^*>t_c-T_\delta$ once we have $f_\gamma\pare{t_c-T_\delta}<g_\gamma\pare{t_c-T_\delta}$ or, equivalently, 
\begin{align}
C_7\gamma^{\frac{1}{4}}+
C_7DR_\gamma^2\delta
<D^{\frac{1}{2}}R_\gamma \delta^{\frac{1}{2}}\pare{1-C_7D^{-\frac{1}{2}}},
\end{align}
where $C_7=C_6^{\frac{1}{2}}$.
Take $D$ large enough such that $1-C_7D^{-\frac{1}{2}}>\frac{1}{2}$.
As we are in $\SCR B_{\delta,\gamma}$, the later inequality is attained if
\begin{align}
C_7\gamma^{\frac{1}{4}}+C_7 D\delta^{\frac{7}{8}}<\frac{1}{2}D^{\frac{1}{2}}\delta^{\frac{3}{4}},
\end{align}
that holds for $\delta$ and $\gamma$ small enough.
Summing up, we proved that, for $D$ large enough (but only depending on the macroscopic parameters) and $\delta$ and $\gamma$ small enough,
\begin{align}
\norm{\ul E^{\pare{1}}_\gamma\pare{t_c-T_\delta}}\le g_\gamma\pare{t_c-T_\delta}=DR_\gamma^2\delta.
\end{align}
Comparing Duhamel's formulas,
using Lemma \ref{az}, and that we are in $\SCR B_{\delta,\gamma}$,
we get
\begin{align}
\norm{\ul X_\gamma\pare{t_c-T_\delta}-\tilde{\ul X}_\gamma\pare{t_c-T_\delta}}\le
C_8\int_0^{t_c-T_\delta}e^{\pare{t_c-T_\delta}\mu}g_\gamma\pare{t_c-T_\delta}ds
\le
C_8 D\delta^{\frac{7}{8}}
\end{align}
for a constant $C_8$.
Then
\begin{align}
\norm{\ul X_\gamma^{\pare{1}}\pare{t_c-T_\delta}}\ge \norm{\tilde{\ul X}_\gamma^{\pare{1}}\pare{t_c-T_\delta}}-C_8 D\delta^{\frac{7}{8}}.
\end{align}
As we are in $\SCR C_{\delta,\gamma}$, the later quantity is bounded by $C_9\delta^{\frac{3}{4}}$;
this implies
\begin{align}
\lim_{\delta\to 0}\liminf_{\gamma\to 0}
\BB P\pare{\norm{\ul X_\gamma^{(k)}\pare{t_c-T_\delta}}>C_9\delta^{\frac{3}{4}}}=1,
\end{align}
that is equivalent to \eqref{fine}.

\begin{proof}[Proof of \eqref{house}]
As $R_\gamma\ge \max\llav{\norm{r_\gamma^{(-1)}},\norm{r_\gamma^{(1)}}}$,
we do not need to worry about the lower bound.
The upper bounds for $a_\gamma^{\pare{1}}$ and $r_\gamma^{\pare{1}}$ follow from Lemma \ref{PPP} and the fact that we know the limiting distribution of $\gamma^{-\frac{1}{2}}\ul X_\gamma^{\pare{k}}\mm\pare{0}$.
Then we are done if we prove that
\begin{align}
\lim_{\delta} \, \liminf_\gamma \, \BB P\corch{b_\gamma^{\pare{1}}< \delta^{-\frac{1}{16}}}=1.
\end{align}
Proceeding as in the proof of Lemma \ref{PPP}, we can reduce the problem to controling
\begin{align}
\sup_{t\ge 0} \, \abs{\gamma^{-\frac{1}{2}}\int_t^\infty e^{-\mu s}d\# M^{\pare{1}}_{\gamma,i}\pare{s}}
\end{align}
for every $i\in\llav{1,2}$ and $\#\in\llav{\mathfrak{R},\mathfrak{I}}$.
We have
\begin{align}
\sup_{t\ge 0} \, \abs{\gamma^{-\frac{1}{2}}\int_t^\infty e^{-\mu s}d\# M^{\pare{1}}_{\gamma,i}\pare{s}}
\le
\abs{\gamma^{-\frac{1}{2}}\int_0^\infty e^{-\mu s}d\# M^{\pare{1}}_{\gamma,i}\pare{s}}+
\sup_{t\ge 0} \, \abs{\gamma^{-\frac{1}{2}}\int_0^t e^{-\mu s}d\# M^{\pare{1}}_{\gamma,i}\pare{s}}.
\end{align}
As $\gamma^{-\frac{1}{2}}\int_0^t e^{-\mu s}d\# M^{\pare{1}}_{\gamma,i}\pare{s}$ is a martingale, we can control last expresion by the use of Doob's maximal inequality and Itô's isometry as we did before.
\end{proof}

\subsection{Proof of Corollary \ref{mozart}}

The result follows once we show that
\begin{itemize}
\item[1.]
$\abs{\gamma^{\frac{\theta}{2}-\frac{1}{2}}\pic{\sigma_{\gamma,i}\pare{t_\theta}, G}-\corch{\hat G\pare{1}\gamma^{\frac{\theta}{2}-\frac{1}{2}} X_{\gamma,i}^{\pare{1}}\pare{t_\theta}+\hat G\pare{-1}\gamma^{\frac{\theta}{2}-\frac{1}{2}} X_{\gamma,i}^{\pare{-1}}\pare{t_\theta}}}\conv{\gamma\to 0}0$ in $\BB P$-probability, and
\item[2.] $\hat G\pare{1}\gamma^{\frac{\theta}{2}-\frac{1}{2}}\underline X_\gamma^{\pare{1}}\pare{t_\theta}+\hat G\pare{-1}\gamma^{\frac{\theta}{2}-\frac{1}{2}}\underline X_\gamma^{\pare{-1}}\pare{t_\theta}$ has the desired limiting distribution.
\end{itemize}
We start by proving the first item.
We have
\begin{align}
&\abs{\gamma^{\frac{\theta}{2}-\frac{1}{2}}\pic{\sigma_{\gamma, i}\pare{t_\theta}, G}-\pare{\hat G\pare{1}\gamma^{\frac{\theta}{2}-\frac{1}{2}} X_{\gamma, i}^{\pare{1}}\pare{t_\theta}+\hat G\pare{-1}\gamma^{\frac{\theta}{2}-\frac{1}{2}} X_{\gamma,i}^{\pare{-1}}\pare{t_\theta}}}
\\[5pt]
&\quad\le\abs{\gamma^{\frac{\theta}{2}-\frac{1}{2}}\pic{\sigma_{\gamma, i}\pare{t_\theta}, G}-\gamma^{\frac{\theta}{2}-\frac{1}{2}}\pic{\sigma_{\gamma, 1}\pare{t_\theta}, \sum_{\abs{k}\leq \gamma^{-\frac{1}{8}\pare{1-\theta}}}\hat G\pare{k} F^{\pare{k}}}}
\\[5pt]
&\quad\blanco{\le}+\abs{\gamma^{\frac{\theta}{2}-\frac{1}{2}}\pic{\sigma_{\gamma, i}\pare{t_\theta}, \sum_{\abs{k}\leq \gamma^{-\frac{1}{8}\pare{1-\theta}}}\hat G\pare{k} F^{\pare{k}}}-\pare{\hat G\pare{1}\gamma^{\frac{\theta}{2}-\frac{1}{2}} X_{\gamma, i}^{\pare{1}}\pare{t_\theta}+\hat G\pare{-1}\gamma^{\frac{\theta}{2}-\frac{1}{2}} X_{\gamma,i}^{\pare{-1}}\pare{t_\theta}}}
\\[5pt]
\label{oki}
&\quad= \gamma^{\frac{\theta}{2}-\frac{1}{2}}\abs{\pic{\sigma_{\gamma, i}\pare{t_\theta}, \sum_{\abs{k}> \gamma^{-\frac{1}{8}\pare{1-\theta}}}\hat G\pare{k} F^{\pare{k}}}}+\gamma^{\frac{\theta}{2}-\frac{1}{2}}\sum_{\substack{\abs{k}\leq \gamma^{-\frac{1}{8}\pare{1-\theta}} \\[3pt] k\neq\pm 1}}\abs{\hat G\pare{k}}\abs{X_{\gamma, i}^{\pare{k}}\pare{t_\theta}}.
\end{align}
%To prove item 1, it is enough to show that both terms of \eqref{oki} converge to $0$, in the limit $\gamma\to 0$.
The first addend in the last expression vanishes because, since $G$ is $C^\infty$, it satisfies $\sum_{k\in \BB Z} \abs{\hat G\pare{k}}\abs{k}^n<\infty$ for every $n\in\BB N$;
the second one, as a  consequence of Proposition \ref{propl}.

To prove the second item, observe that, by decomposing $\hat G\pare{1}$ in real and imaginary parts, we get
\begin{align}\label{jhb}
\begin{aligned}
&\hat G\pare{1}\gamma^{\frac{\theta}{2}-\frac{1}{2}}X_{\gamma,i}^{\pare{1}}\pare{t_\theta}+\hat G\pare{-1}\gamma^{\frac{\theta}{2}-\frac{1}{2}} X_{\gamma,i}^{\pare{-1}}\pare{t_\theta}
\\[5pt]
&\quad =\mathfrak{R}\pare{\hat G\pare{1}}\gamma^{\frac{\theta}{2}-\frac{1}{2}}\pare{X_{\gamma,i}^{\pare{1}}\pare{t_\theta}
+X_{\gamma,i}^{\pare{-1}}\pare{t_\theta}}
+\mathtt i \mathfrak{I}\pare{\hat G\pare{1}}\gamma^{\frac{\theta}{2}-\frac{1}{2}}\pare{X_{\gamma,i}^{\pare{1}}\pare{t_\theta}-X_{\gamma,i}^{\pare{-1}}\pare{t_\theta}}
\\[5pt]
&\quad=2\mathfrak{R}\pare{\hat G\pare{1}}\gamma^{\frac{\theta}{2}-\frac{1}{2}}\mathfrak{R}\pare{X_{\gamma,i}^{\pare{1}}\pare{t_\theta}}
+2\mathfrak{I}\pare{\hat G\pare{1}}\gamma^{\frac{\theta}{2}
-\frac{1}{2}}\mathfrak{I}\pare{X_{\gamma,i}^{\pare{1}}\pare{t_\theta}}.
\end{aligned}
\end{align}
If we call $Z_i$ the limiting distribution of $\gamma^{\frac{\theta}{2}-\frac{1}{2}}X_{\gamma,i}^{\pare{1}}\pare{t_\theta}$ given in Theorem \ref{1},
last expression converges in distribution to  $\pic{2\corch{Z_i}_{\mathfrak{R}}\cos\pare{2\pi\cdot}+2\corch{Z_i}_{\mathfrak{I}}\sin\pare{2\pi\cdot}, G}=\pic{2\abs{Z_i}\cos\pare{2\pi\cdot+\Phi_1}, G}$,
where $\Phi_i=-\arctan\frac{\corch{Z_i}_{\mathfrak{I}}}{\corch{Z_i}_{\mathfrak{R}}}$.
In the previous identity, we used that $c_1\cos\xi+c_2\sin\xi=\sqrt{c_1^2+c_2^2}\cos\pare{\zeta-\arctan \frac{c_2}{c_1}}$.
The fact that $\Phi_i$ is uniform follows from the rotational invariance of the model.
We can conclude from the fact that the square of the euclidean norm of a Gaussian vector has gamma distribution.

%A similiar argument holds for the limit in distribution of $\hat G\pare{1}\gamma^{\frac{\theta}{2}-\frac{1}{2}}X_{\gamma, 2}^{\pare{1}}\pare{t_\theta}+\hat G\pare{-1}\gamma^{\frac{\theta}{2}-\frac{1}{2}} X_{\gamma, 2}^{\pare{-1}}\pare{t_\theta}$, so we get item 2.

\subsection{Proof of theorem \ref{hi77}}

Conditions $\mu_1^{\pare{k}}+\mu_2^{\pare{k}}=\tr^{\pare{k}}<0$ and $\mu_1^{\pare{k}}\mu_2^{\pare{k}}=\det^{\pare{k}}>0$ are necessary and sufficient for linear stability of the $k$-th Fourier mode.
This is obvious if the eigenvalues are real; if they have nonzero imaginary parts, the equivalence follows from identities $\tr^{\pare{k}}=2\mathfrak{R}\pare{\mu_1^{\pare{k}}}$ and $\det^{\pare{k}}=\abs{\mu_1^{\pare{k}}}^2$ (as $A^{\pare{k}}$ is real for every $k$, complex eigenvalues come in conjugate pairs).
In our case, these conditions read
\begin{align}\label{cond1}
\alpha_1 e^{-\tilde \tau_1k^2}+\alpha_2 e^{-\tilde \tau_2k^2}<2
\end{align}
and
\begin{align}\label{cond2}
\pare{\alpha_1e^{-\tilde \tau_1k^2}-1}\pare{1-\alpha_2e^{-\tilde \tau_2k^2}}<\tanh\pare{\lambda\beta_1}\tanh\pare{\lambda\beta_2}.
\end{align}
Also condition
\begin{align}\label{cond3}
\tn{det}^{\pare{k}}<0\iff
\pare{\alpha_1e^{-\tilde \tau_1k^2}-1}\pare{1-\alpha_2e^{-\tilde \tau_2k^2}}>\tanh\pare{\lambda\beta_1}\tanh\pare{\lambda\beta_2}
\end{align}
is a sufficient condition for linear  instability of the $k$-th Fourier mode (in particular, it implies that the eigenvalues are real).
Necessary and sufficient conditions for Turing instability are conditions \eqref{cond1} and \eqref{cond2} for $k=0$ and condition \eqref{cond3} for some $k\neq 0$.
As the sufficiency is obvious, we only prove the necessity.
Suppose then Turing instability occurs.
Conditions \eqref{cond1} and \eqref{cond2} for $k=0$ follow immediately.
Let $k_0\neq 0$ be one of the values where we have linear instability: $\mathfrak{R}\pare{\mu_1^{\pare{k_0}}}>0$.
Condition \eqref{cond1} for $k=0$ implies the same condition for any $k$ and, in particular, for $k_0$: $\mathfrak{R}\pare{\mu^{\pare{k_0}}_1}+\mathfrak{R}\pare{\mu^{\pare{k_0}}_2}=\tr^{\pare{k_0}}<0$.
Then 
$\mathfrak{R}\pare{\mu^{\pare{k_0}}_2}<0$
so $\mu^{\pare{k_0}}_1$ and $\mu^{\pare{k_0}}_2$ are real (their real parts are different so they are not pair conjugate), giving \eqref{cond3} for $k_0$.

\begin{itemize}
\item[\it{(i)}]
Assume Turing instability occurs.
Suppose $\alpha_1\ge \alpha_2$.
Condition \eqref{cond1} for $k=0$ implies $\alpha_2<1$.
If $\alpha_1\le 1$, then $\pare{\alpha_1e^{-\tilde \tau_1k^2}-1}\pare{1-\alpha_2e^{-\tilde \tau_2k^2}}<0$ for every $k$, so condition \eqref{cond3} cannot occur for any $k$; then $\alpha_1>1$.
Finally, suppose $\tau_2\le \tau_1$.
Think $k$ as a real variable and define $f\pare{k}\defi \pare{\alpha_1e^{-\tilde \tau_1k^2}-1}\pare{1-\alpha_2e^{-\tilde \tau_2k^2}}$.
We are done if we prove that $f\pare{k}\le f\pare{0}$ for every $k$ because, in this case, conditions \eqref{cond3} cannot occur for any $k$.
As $f$ is symmetric, it is enough to prove that it is decreasing for $k>0$.
Condition $f'\pare{k}\le 0$ is equivalent to
\begin{align}\label{c}
\tilde \tau_1+\tilde \tau_2\leq\frac{\tilde \tau_1e^{\tilde \tau_2 k^2}}{\alpha_2}+\frac{\tilde \tau_2e^{\tilde \tau_1 k^2}}{\alpha_1}.
\end{align}
Since the right-hand side of \eqref{c} is increasing and the left-hand side is constant,    the later condition holds if and only if
\begin{align}\label{d}
\tilde \tau_1+\tilde \tau_2\leq \frac{\tilde \tau_1}{\alpha_2}+\frac{\tilde \tau_2}{\alpha_1},
\end{align}
inequality that follows from ous assumptions.
The case $\alpha_1<\alpha_2$ is proved in a similar way.

\item[\it{(ii)}]
We will prove something stronger: \eqref{cond1} and \eqref{cond2} for $k=0$ and \eqref{cond3} for $k=1$.
As $\alpha_i<\beta_i$, condition \eqref{cond1} for $k=0$ follows from the hypothesis $\beta_1+\beta_2\le 2$.

Let $g\pare{\lambda}=\pare{\alpha_1-1}\pare{1-\alpha_2}$ with the domain extended also to $\lambda=0$.
As $\beta_2<1<\beta_1$, $g\pare{0}=\pare{\beta_1-1}\pare{1-\beta_2}>0$.
For $\lambda>0$, $g'\pare{\lambda}\le 0$ if and only if
\begin{align}\nonumber
\frac{\pare{\alpha_1-1}\alpha_2}{\alpha_1\pare{1-\alpha_2}}\le \frac{\beta_1}{\beta_2}\frac{\tanh\pare{\lambda\beta_1}}{\tanh\pare{\lambda\beta_2}}.
\end{align}
As $\beta_1>\beta_2$, the right-hand side is strictly larger than $1$, so the negativity of the derivative follows if the left-hand side is smaller than $1$ or, equivalently, if
\begin{align}\nonumber
2\le \frac{\corch{\cosh\pare{\lambda\beta_1}}^2}{\beta_1}+\frac{\corch{\cosh\pare{\lambda\beta_2}}^2}{\beta_2}.
\end{align}
As the hyperbolic cosine is larger than $1$, the later inequality follows from inequality $2\le \frac{1}{\beta_1}+\frac{1}{\beta_2}$, that follows form inequality $\beta_1+\beta_2\le 2$.
Once we know the function $\pare{\alpha_1-1}\pare{1-\alpha_2}$ is decreasing (and using that it is strictly positive at $\lambda=0$), we can define the parameter $\lambda^*=\lambda^*\pare{\beta_1,\beta_2}$ as the unique positive solution of identity
\begin{align}\nonumber
\pare{\alpha_1-1}\pare{1-\alpha_2}=\tanh(\beta_1\lambda)\tanh(\beta_2\lambda).
\end{align}

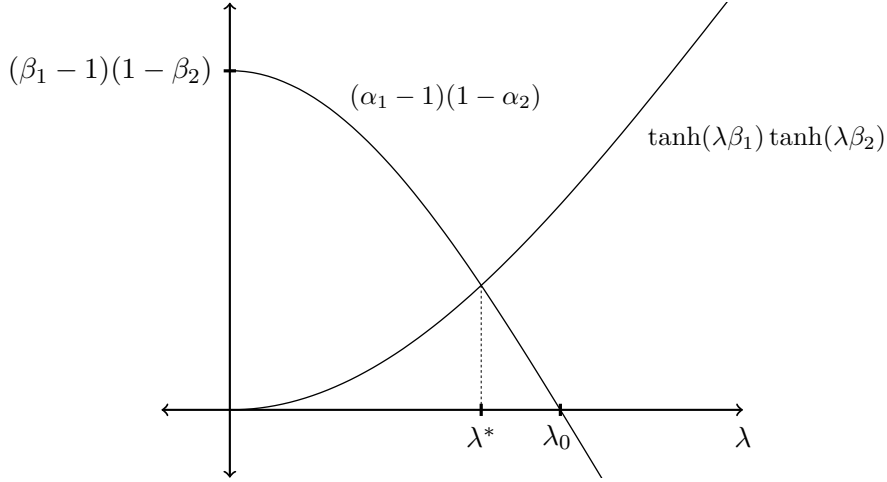
\begin{figure}[H]
\centering
\begin{tikzpicture}[thick,yscale=0.9, xscale=0.9]
\clip(-3.5,-1) rectangle (10,6);
\draw[<->] (-1,0) -- (7.5,0)  {};
\draw[<->] (0,-1) -- (0,6)  {};
\draw[-] [line width=1.2pt] (-0.09,4.995) -- (0.09,4.995)  {};
\draw[-] [line width=1.2pt] (4.83,-0.09) -- (4.83,0.09)  {};
\draw[-] [line width=1.2pt] (3.673,-0.09) -- (3.673,0.09)  {};
\draw[-] [line width=0.3pt] (3.673,0) -- (3.673,1.832)  [dash pattern=on 1pt off 1pt,domain=-4:4];
\draw  (7.5,-0.4) node {$\lambda$};
\draw  (4.78,-0.4) node {$\lambda_0$};
\draw  (3.673,-0.4) node {$\lambda^*$};
\draw  (3.15,4.6) node {\small $\pare{\alpha_1-1}\pare{1-\alpha_2}$};
\draw  (7.85,4) node {\small $\tanh\pare{\lambda\beta_1}\tanh\pare{\lambda\beta_2}$};
\draw  (-1.75,4.995) node {$\pare{\beta_1-1}\pare{1-\beta_2}$};
\draw[line width=0.5pt,xscale=11,yscale=20,domain=-0:5,samples=500,variable=\x] plot ({\x},{  (tanh(\x*1.1))*(tanh(\x*0.8)) });
\draw[line width=0.5pt,xscale=11,yscale=20,domain=-0:0.5,samples=50,variable=\x] plot ({\x},{  
((1.5)/((cosh(1.5*\x))^2)-1)*(1-(0.5)/((cosh(0.5*\x))^2))
 });
\end{tikzpicture}
\caption{The parameter $\lambda^*$.} \label{figure6}
\end{figure}

Let $\alpha_i^*=\left. \alpha_i\right|_{\lambda=\lambda^*}$.
Think $k$ as a continuous variable and let $f^*\pare{k}=\pare{\alpha^*_1e^{-\tilde \tau_1k^2}-1}\pare{1-\alpha^*_2e^{-\tilde \tau_2k^2}}$.
As $f^*$ is symmetric, we only need to analyze the case $k>0$.
Condition ${f^*}'\pare{k}>0$ is equivalent to
\begin{align}\label{uuuu}
\frac{\tilde \tau_1 e^{\tilde \tau_2 k^2}}{\alpha_2^{*}} +\frac{\tilde \tau_2 e^{\tilde \tau_1 k^2}}{\alpha_1^{*}}   < \tilde \tau_1+\tilde \tau_2.
\end{align}
We will ask for $f^*$ to be strictly increasing for $0<k\leq 1$.
As the functions $e^{\tilde \tau_1 k^2}$ and $e^{\tilde \tau_2 k^2}$ are increasing in $k$, and taking into consideration \eqref{uuuu}, a sufficient condition for this is
\begin{align}\nonumber
\frac{\tilde \tau_1 e^{\tilde \tau_2}}{\alpha_2^{*}} +\frac{\tilde \tau_2 e^{\tilde \tau_1}}{\alpha_1^{*}}   < \tilde \tau_1+\tilde \tau_2
\end{align}
or, equivalently,
\begin{align}\label{papa}
\frac{\pare{\alpha_2^*}^{-1}e^{\tilde \tau_2}-1}{\tilde \tau_2}
< \frac{1-\pare{\alpha_1^*}^{-1}e^{\tilde \tau_1}}{\tilde \tau_1}.
\end{align}
Chose $\tau_1$ and $\tau_2$ in order this inequality holds (this is always possible as the right-hand side goes to infinity as $\tau_1\downarrow 0$).
Under these assumptions, we have 
\begin{align}\nonumber
f^*\pare{0}&=\pare{\alpha_1^*-1}\pare{1-\alpha_2^*}=\tanh(\beta_1\lambda^*)\tanh(\beta_2\lambda^*)
\\[5pt]
&<f^*\pare{1}=\pare{\alpha_1^*e^{-\tilde \tau_1}-1}\pare{1-\alpha_2^*e^{-\tilde \tau_2}}.
\end{align}
By continuity, we can conclude after choosing $\lambda>\lambda^*$ as a perturbation of $\lambda^*$.

\item[\it{(iii)}]
In the previous item, conditions  $\beta_2<1<\beta_1$ and $\beta_1+\beta_2<2$ were used only for the good definition of $\lambda^*$.
Observe that if $\beta_2$ is close to $1$ also $\beta_1$ is in the sense that $\beta_2>1-\delta$ implies $\beta_1<1+\delta$.
As $\alpha_2<1$, $\pare{\alpha_1-1}\pare{1-\alpha_2}$ vanishes only if $\alpha_1-1$ does; let $\lambda_0$ be the unique value where it occurs (it is unique because $\alpha_1$ is decreasing in $\lambda$).
As $\lambda_0$ vanishes as $\beta_2$ goes to $1$ (because also $\beta_1$ goes to $1$), and as domination $\lambda^*\le\lambda_0$ holds, we conclude $\lambda^*$ also vanishes in this case.

\begin{figure}[H]
\centering
\begin{tikzpicture}[thick,yscale=0.9, xscale=0.9]
\path [fill=gray!35!] (2.85,2.85) -- (2.85,2) -- (3.7,2);
\draw[<->] (-1,0) -- (6.7,0)  {};
\draw[<->] (0,-1) -- (0,6.7)  {};
\draw[-] [line width=1.2pt] (-0.09,5.7) -- (0.09,5.7)  {};
\draw[-] [line width=1.2pt] (-0.09,2) -- (0.09,2)  {};
\draw[-] [line width=1.2pt] (-0.09,2.85) -- (0.09,2.85)  {};
\draw[-] [line width=1.2pt] (5.7,-0.09) -- (5.7,0.09)  {};
\draw[-] [line width=1.2pt] (2.85,-0.09) -- (2.85,0.09)  {};
\draw[-] [line width=0.3pt] (0,5.7) -- (5.7,0)  [dash pattern=on 1pt off 1pt,domain=-4:4];
\draw[-] [line width=0.3pt] (0,2.85) -- (2.85,2.85)  [dash pattern=on 1pt off 1pt,domain=-4:4];
\draw[-] [line width=0.3pt] (2.85,0) -- (2.85,2.85)  [dash pattern=on 1pt off 1pt,domain=-4:4];
\draw[-] [line width=0.3pt] (0,2) -- (3.7,2)  [dash pattern=on 1pt off 1pt,domain=-4:4];
\draw  (6.8,-0.4) node {$\beta_1$};
\draw  (5.7,-0.4) node {$2$};
\draw  (2.85,-0.4) node {$1$};
\draw  (-0.4,6.7) node {$\beta_2$};
\draw  (-0.3,5.7) node {$2$};
\draw  (-0.3,2.85) node {$1$};
\draw  (-0.73,2) node {$1-\delta$};
\end{tikzpicture}
\caption{The grey region is where we can guarantee unimodular Turing instability.} \label{figure1}
\end{figure}
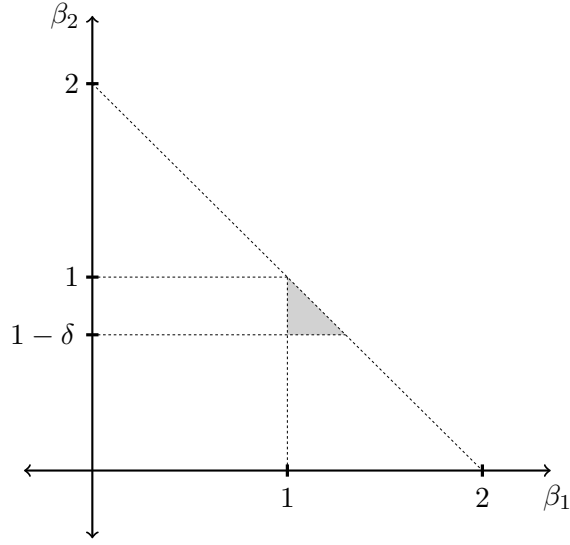

The proof is similar than the one of the previous item.
In this case, we will get the stronger property $f^*\pare{k}<f^*\pare{0}=\tanh\pare{\lambda\beta_1}\tanh\pare{\lambda\beta_2}<f^*\pare{1}$ for every $k\ge 2$. 
For $k>0$, ${f^*}'\pare{k}=0$ if and only if
\begin{align}\nonumber
\frac{\tilde \tau_1 e^{\tilde \tau_2 k^2}}{\alpha_2^{*}} +\frac{\tilde \tau_2 e^{\tilde \tau_1 k^2}}{\alpha_1^{*}}   = \tilde \tau_1+\tilde \tau_2.
\end{align}
As the left-hand side is increasing, this can occur at most in one value of $k>0$.
As $f^*\pare{0}>0$ and $f^*\pare{k}\conv{k\to \infty}-1$, $f^*$ has only one positive root $\hat k$.
Then sufficient conditions for unimodular Turing instability are ${f^*}'\pare{k}>0$ for $0<k\le 1$ and $\hat k<2$.
\begin{figure}[H]
\centering
\begin{tikzpicture}[thick,yscale=0.9, xscale=0.9]
\clip(-4.5,-1) rectangle (10,6.2);
\draw[<->] (-1,0) -- (7.5,0)  {};
\draw[<->] (0,-1) -- (0,6)  {};
\draw[-] [line width=1.2pt] (-0.09,2.107) -- (0.09,2.107)  {};
\draw[-] [line width=1.2pt] (5.4,-0.09) -- (5.4,0.09)  {};
\draw[-] [line width=1.2pt] (2.7,-0.09) -- (2.7,0.09)  {};
\draw[-] [line width=0.3pt] (2.7,0) -- (2.7,4.945)  [dash pattern=on 1pt off 1pt,domain=-4:4];
\draw[-] [line width=1.2pt] (5,-0.09) -- (5,0.09)  {};
\draw  (4.87,-0.4) node {$\hat k$};
\draw  (7.5,-0.4) node {$k$};
\draw  (5.4,-0.4) node {$2$};
\draw  (2.7,-0.4) node {$1$};
\draw  (-0.65,6) node {\small $f^*\pare{k}$};
\draw  (-2.2,2.107) node {$\tanh\pare{\lambda\beta_1}\tanh\pare{\lambda\beta_2}$};
\draw[line width=0.5pt,xscale=5,yscale=6.5,domain=0:1.5,samples=50,variable=\x] plot ({\x},{  
10*(exp(0.5*(1-\x^2))-1)*(1-0.95*exp(-0.5*(\x^2)))
 });
\end{tikzpicture}
\caption{The function $f^*$.} \label{figure7}
\end{figure}
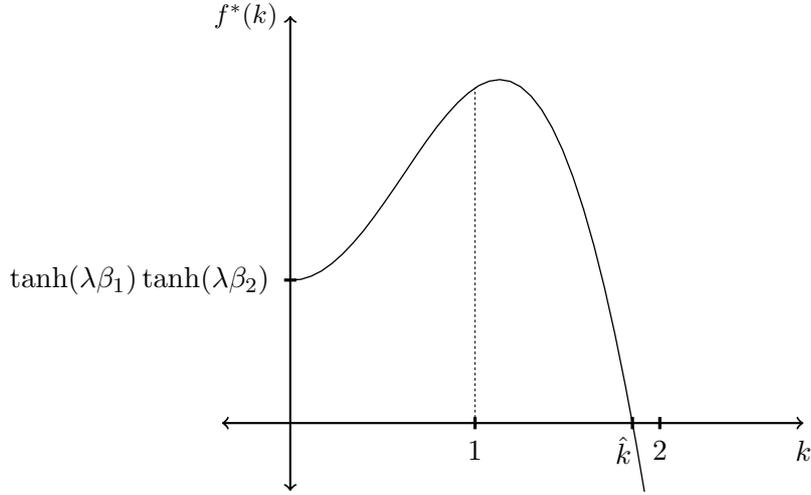
Recall the first one follows from condition \eqref{papa}.
For the second one, observe that, as $1-\alpha_2^* e^{-\tilde \tau_2k^2}$ never vanishes, $f^*\pare{k}=0$ if and only if $\alpha_2^* e^{-\tilde \tau_2k^2}-1=0$ or, equivalently,
\begin{align}\label{igig}
\tilde \tau_1> \frac{\log \alpha_1^*}{4}.
\end{align}
We are done if we see that conditions \eqref{papa} and \eqref{igig} can be simultaneously satisfied.

Choose $\tilde \tau_1=\frac{\log \alpha_1^*}{3}$ so condition \eqref{igig} is automatically fulfilled.
%Observe that, under this definition, $\tilde \tau_1>0$.
%To prove that, define $\lambda_0$ as the solution of equation $\alpha_1^*=1$.
%As $1-\alpha_2^*$ does not vanish, $\lambda_0$ is also the unique root of $\pare{\alpha_1^*-1}\pare{1-\alpha_2^*}$ (recall the later function is decreasing in $\lambda$), so $\lambda_0>\lambda^*$ and $\alpha^*>1$.
Replacing in the right-hand side of \eqref{papa}, we get
\begin{align}\nonumber
3\frac{1-\pare{\alpha_1^*}^{-2/3}}{\log \alpha_1^*},
\end{align}
that converges to $2$ as $\alpha_1^*\downarrow 1$.
As $\alpha_1^*<\beta_1$, it is enough to take $\beta_2$ close enough to $1$ (and then also $\beta_1$ close to $1$) for the later quantity to be larger than say $2-\frac{1}{2}$.

To analyze the left-hand side of \eqref{papa}, we observe that the derivative with respect to $\tilde \tau_2$ vanishes at the unique positive solution of
\begin{align}\nonumber
\alpha_2^*=e^{\tilde \tau_2}\pare{1-\tilde \tau_2},
\end{align}
solution that we call $\hat{\tilde c}_2$.
As the right-hand side of this expression is $1$ at $\tilde \tau_2=0$, and as it is decreasing, $\hat{\tilde c}_2$ tends to zero as $\alpha_2^*\uparrow 1$.
As $\hat{\tilde c}_2$ satisfies
\begin{align}\nonumber
\pare{\alpha_2^*}^{-1} e^{\hat{\tilde c}_2}-1=\pare{\alpha_2^*}^{-1}\hat{\tilde c}_2e^{\hat{\tilde c}_2},
\end{align}
we can replace in the left-hand side of \eqref{papa} and obtain that the value of the minimum is
$\pare{\alpha_2^*}^{-1}e^{\hat{\tilde c}_2}$, that tends to $1$ as $\alpha_2^*\uparrow 1$ ($\alpha_2^*$ goes to one because, as mentioned before, $\lambda^*$ vanishes).

We can conclude after taking $\tilde \tau_2=\hat{\tilde c}_2$ and $\lambda>\lambda^*$ a perturbation of $\lambda^*$.
\end{itemize}

\section{Appendix}

\begin{lemma}\label{az}
There exists $ C  = C  \pare{\MP}$ such that
\begin{align}
\norm{e^{tA^{\pare{k}}}}\le  C  e^{t\mu}
\end{align}
for $k=\pm 1$, and
\begin{align}
\norm{e^{tA^{\pare{k}}}}\le  C  e^{\frac{1}{2}t\mathfrak{R} \pare{\mu_1^{\pare{k}}}}
\end{align}
for $k\neq \pm 1$.
\end{lemma}
\begin{proof}[Proof of Lemma \ref{az}]
We first analyze the case $k=1$
(the case $k=-1$ is the same as $A^{\pare{1}}=A^{\pare{-1}}$).
Let $S^{\pare{1}}$ be the matrix with columns $\ul v_1^{\pare{1}}$ and $\ul v_2^{\pare{1}}$.
We have
\begin{align}
\norm{e^{tA^{\pare{k}}}}\le \norm{S^{\pare{1}}} \norm{\pare{
\begin{array}{cc}
e^{t\mu_1^{\pare{k}}} &0
\\[5pt]
0 & e^{t\mu_2^{\pare{k}}}
\end{array}
}} \norm{\pare{S^{\pare{1}}}^{-1}}=
\norm{S^{\pare{1}}}\norm{\pare{S^{\pare{1}}}^{-1}} e^{t\mu}.
\end{align}

The case $k\neq \pm 1$ will be decomposed in two sub-cases.
There exists $ C  _1= C  _1\pare{\MP}$ such that $\abs{\Dis^{\pare{k}}}\ge 2\tanh\pare{\lambda\beta_1}\tanh\pare{\lambda\beta_2}$ for every $k$ such that $\abs{k}\ge  C  _1$.

We first analyze the sub-case $\abs{k}\ge C_1$.
In this case, the matrix is diagonalizable;
the difference with the case $k=\pm 1$ is that we have to control coefficients in the variable $k$.
Let $S^{\pare{k}}$ be the matrix with columns $\ul v_1^{\pare{k}}$ and $\ul v_2^{\pare{k}}$.
We have
\begin{align}
\norm{e^{tA^{\pare{k}}}}\le
\norm{S^{\pare{k}}}
\norm{\pare{
\begin{array}{cc}
e^{t\mu_1^{\pare{k}}} & 0 \\[5pt]
0 & e^{t\mu_2^{\pare{k}}}
\end{array}
}}
\norm{\pare{S^{\pare{k}}}^{-1}}=
\norm{S^{\pare{k}}}
e^{t\mathfrak{R}\pare{\mu_1^{\pare{k}}}}
\norm{\pare{S^{\pare{k}}}^{-1}}.
\end{align}
There exists $C_2=C_2\pare{\MP}$ such that
$\max\llav{\norm{\ul v^{\pare{k}}_1}_1,\norm{\ul v^{\pare{k}}_2}_1}\le C_2$, so $\norm{S^{\pare{k}}}\le C_2$.
As $\pare{S^{\pare{k}}}^{-1}=\frac{1}{\det S^{\pare{k}}}\tilde S^{\pare{k}}$ with $\tilde S^{\pare{k}}$ obtained from $S^{\pare{k}}$ after rearranging the coefficients and changing the signs of some of them,
there exists $ C  _3= C  _3\pare{\text U}$ such that 
\begin{align}
\norm{\pare{S^{\pare{k}}}^{-1}}\le
 C  _3\abs{\det S^{\pare{k}}}^{-1}=
 C  _3\frac{1}{\tanh\pare{\lambda\beta_2}\abs{\sqrt{\Dis^{\pare{k}}}}}\le
  C  _3\frac{1}{\tanh\pare{\lambda\beta_2} \sqrt{2\tanh\pare{\lambda\beta_1}\tanh\pare{\lambda\beta_2}} }.
\end{align}

We finally analyze the sub-case $k\neq \pm 1$ and $\abs{k}<C_1$.
If $\Dis^{\pare{k}}\neq 0$, we can proceed as in the case $k=1$ as we have to control only a finite number of $k$'s.
If $\Dis^{\pare{k}}=0$, we have $\mu_1^{\pare{k}}=\mu_2^{\pare{k}}=\mu^{\pare{k}}$.
The matrix $A^{\pare{k}}$ is not diagonalizable but equivalent to a triangular matrix
\begin{align}
T^{\pare{k}}=\pare{
\begin{array}{cc}
\mu^{\pare{k}} & a^{\pare{k}}
\\[5pt]
0 & \mu^{\pare{k}}
\end{array}
}
\end{align}
via conjugating by orthogonal matrices.
Then there exists $ C  _4= C  _4\pare{\text U}$ such that
\begin{align}
\norm{e^{tA^{\pare{k}}}}\le  C  _4\norm{e^{tT^{\pare{k}}}}= C  _4\norm{
\pare{\begin{array}{cc}
e^{t\mu^{\pare{k}}} & a^{\pare{k}}te^{t\mu^{\pare{k}}}
\\[5pt]
0 & e^{t\mu^{\pare{k}}}
\end{array}}
}=
 C  _4\pare{e^{t\mu^{\pare{k}}}+\abs{a^{\pare{k}}}te^{t\mu^{\pare{k}}}}.
\end{align}
We conclude by observing that there exists $ C  _5= C  _5\pare{\MP}$ such that $e^{t\mu^{\pare{k}}}+\abs{a^{\pare{k}}}te^{t\mu^{\pare{k}}}\le  C  _5 e^{\frac{1}{2}t\mu^{\pare{k}}}$ (as we have only finite cases to consider, $\abs{a^{\pare{k}}}$ can be bounded by a constant that depends only on the MP).
\end{proof}

\begin{proposition}\label{himan}
For $M\in\BB N$ and $j\in\llav{1,\ldots,M}$, let $I_j\defi \left[\frac{j-1}{M},\frac{j}{M}\right)$.
For $f\in L^1\pare{\BB T,\BB R}$ and $M\in\BB N$, let $f_M$ be the function such that, for every $j\in\llav{1,\ldots,M}$, takes the value $M\int_{I_j}f$ in the interval $I_j$ ($f_M$ is a piece-wise constant approximation of $f$).
Let $f\in C\pare{\BB T,\BB R}$ such that $\int_0^1f_M^2\conv{M\to\infty}\int_0^1f^2$.
Let $\pare{\sigma_i}_{i=0}^{N-1}$ be an independent family with distribution $P\corch{\sigma_i=1}=P\corch{\sigma_i=-1}=\frac{1}{2}$.
Then
\begin{align}\label{gibbs}
Y^{\pare{N}}\defi\frac{1}{\sqrt N}\sum_{i=0}^{N-1}f\pare{\frac{i}{N}}\sigma_i
\end{align}
converges in distribution to $N\pare{0,\int_0^1 f^2}$.
\end{proposition}

To prove lemma \ref{gibbs}, we need the following one.

\begin{lemma}\label{565}
If $f\in L^1\pare{\BB T,\BB R}$ is such that $f=f_M$ for some $M\in\BB N$, the assertion of lemma \ref{gibbs} holds.
\end{lemma}
\begin{proof}[Proof of lemma \ref{565}]
Let $\Lambda_j^{\pare{N}}\defi\pare{NI_j}\cap \BB Z$.
We first see that
\begin{align}\label{evans}
\tilde Y^{\pare{N}}\defi \frac{1}{\sqrt{M}}\sum_{j=1}^M f\pare{\frac{j-1}{M}}\frac{1}{\sqrt{\abs{\Lambda_j^{\pare{N}}}}}
\sum_{i\in\Lambda_j^{\pare{N}}}\sigma_i
\end{align}
converges weakly to $N\pare{0,\int_0^1f^2}$.
Call $X^{\pare{N}}_j\defi \frac{1}{\sqrt{\abs{\Lambda_j^{\pare{N}}}}}
\sum_{i\in\Lambda_j^{\pare{N}}}\sigma_i$.
For every $N$, the family $\llav{X_j^{\pare{N}}}_{j}$ is independent.
Also $X_j^{\pare{N}}$ converges weakly to $N\pare{0,1}$ for every $j$.
Then the random vector $\pare{X_j^{\pare{N}}}_{j}$ converges weakly to the random vector $\pare{X_j}_j\sim N\pare{\ubar 0,\mbox{Id}_M}$.
Then the random variable \eqref{evans} converges weakly to $\frac{1}{\sqrt{M}}\sum_{j=1}^M f\pare{\frac{j-1}{M}}X_j\sim N\pare{0,\int_0^1f^2}$.

Using that $f$ is bounded, that $V\pare{X_j^{\pare{N}}}=1$, and that
$\sqrt{\frac{\abs{\Lambda_j^{\pare{N}}}}{N}}\conv{N\to\infty}\frac{1}{\sqrt M}$,
one can see that $V\pare{\tilde Y^{\pare{N}}-Y^{\pare{N}}}\conv{N\to\infty}0$;
then, for every $\tilde\delta>0$,
\begin{align}\label{amstrong}
P\pare{\abs{\tilde Y^{\pare{N}}-Y^{\pare{N}}}>\tilde\delta}\conv{N\to\infty}0.
\end{align}
Let $G_{\int_0^1f^2}$ be the Gaussian probability with zero mean and variance  $\int_0^1f^2$.
For $h:\BB R\to\BB R$ bounded and uniformly continuous, we have to prove that
\begin{align}
\abs{E\pare{h\pare{Y^{\pare{N}}}}-G_{\int_0^1f^2}\pare{h}}\conv{N\to\infty}0.
\end{align}
(Recall that weak convergence of probabilities is equivalent to convergence of the expectations against bounded uniformly continuous functions.)
We have already proved that 
\begin{align}
\abs{E\pare{h\pare{\tilde Y^{\pare{N}}}}-G_{\int_0^1f^2}\pare{h}}\conv{N\to\infty}0,
\end{align}
so we only need to prove
\begin{align}\label{nichols}
\abs{E\pare{h\pare{Y^{\pare{N}}}}-E\pare{h\pare{\tilde Y^{\pare{N}}}}}\conv{N\to\infty}0.
\end{align}
Fix $\varepsilon>0$ and take $\delta>0$ such that $\abs{h\pare{y}-h\pare{x}}<\varepsilon$ whenever $\abs{y-x}\le \delta$.
The quantity to control in \eqref{nichols} is bounded by
\begin{align}
\begin{aligned}
&E\pare{\abs{h\pare{Y^{\pare{N}}}-h\pare{\tilde Y^{\pare{N}}}}\11\llav{\abs{Y^{\pare{N}}-\tilde Y^{\pare{N}}}>\delta}}
\\[5pt]
&+E\pare{\abs{h\pare{Y^{\pare{N}}}-h\pare{\tilde Y^{\pare{N}}}}\11\llav{\abs{Y^{\pare{N}}-\tilde Y^{\pare{N}}}\le\delta}}.
\end{aligned}
\end{align}
The first addend goes to zero because of \eqref{amstrong}; the second one is bounded by $\varepsilon$.
As $\varepsilon$ is arbitrary, we can conclude.
\end{proof}

\begin{proof}[Proof of lemma \ref{gibbs}]
Let $f_M$ be the discretized version of $f$ and
\begin{align}
Y^{\pare{N}}_M\defi\frac{1}{\sqrt N}\sum_{i=0}^{N-1}f_M\pare{\frac{i}{N}}\sigma_i.
\end{align}
From Chebyshev inequality, there is a constant depending only on $f$ such that, for every $\tilde \delta>0$,
\begin{align}
P\pare{\abs{Y^{\pare{N}}-Y_M^{\pare{N}}}>\tilde\delta}\le \frac{C}{\tilde\delta^2M^2}
\end{align}
(observe that this bound is uniform in $N$).
Let $h:\BB R\to\BB R$ bounded and uniformly continuous.
We have to prove that
\begin{align}\label{casey}
\abs{E\pare{h\pare{Y^{\pare{N}}}}-G_{\int_0^1 f^2}h}\conv{N\to\infty}0.
\end{align}
Fix $\varepsilon>0$ and let $\delta>0$ be such that $\abs{h\pare{y}-h\pare{x}}<\varepsilon$ whenever $\abs{y-x}\le\delta$.
Take $M$ such that $\abs{G_{\int_0^1 f_M^2}h-G_{\int_0^1 f^2}h}<\varepsilon$ and 
$\frac{C}{\delta^2 M^2}< \frac{\varepsilon}{\norm{h}_\infty}$.
The quantity to control in \eqref{casey} is bounded by
\begin{align}
E\pare{\abs{h\pare{Y^{\pare{N}}}-h\pare{Y_M^{\pare{N}}}}}
+\abs{E\pare{h\pare{Y_M^{\pare{N}}}}- G_{\int_0^1f_M^2}h }+\varepsilon.
\end{align}
Multiply by $1=\11\llav{\abs{Y^{\pare{N}}-Y_M^{\pare{N}}}>\delta}+\11\llav{\abs{Y^{\pare{N}}-Y_M^{\pare{N}}}\le\delta}$ inside the first expectation to get the upper bound $2\varepsilon$ for it.
The second addend goes to zero as $N$ goes to infinity because of Lemma \ref{565}.
We conclude as $\varepsilon$ is arbitrary.
\end{proof}

\begin{lemma}\label{PPP}
There exists a constant $C$ such that, for every $\delta>0$,
\begin{align}\label{re}
\BB P\pare{\sup_{t\in [0,\infty)}\norm{\int_0^t e^{\pare{t-s} A^{\pare{k}}}\ul M_\gamma^{\pare{k}}\pare{ds}}\leq k^2 \gamma^{\frac{1}{2}-\delta}\quad\forall k\neq\pm 1}\geq 1-C\gamma^{2\delta}
\end{align}
\begin{align}\label{4}
\BB P\pare{\sup_{t\in [0,\infty)}\norm{\int_0^t e^{-sA^{\pare{k}}}\ul M_\gamma^{\pare{k}}\pare{ds}}\leq \gamma^{\frac{1}{2}-\delta}, \;k=\pm 1}\geq1-  C\gamma^{2\delta}
\end{align}
\begin{align}\label{rew}
\BB P\pare{\norm{\ul{X}_\gamma^{\pare{k}}\pare{0}}\leq k^2 \gamma^{\frac{1}{2}-\delta}\quad\forall k\in\BB Z}\geq 1-C\gamma^{2\delta}.
\end{align}
\end{lemma}
\begin{proof}
We start by proving \eqref{re}. Observe that the proof follows once we show that there exists a constant $C$ such that, for all $T>0$,
\begin{align}\label{poiyl}
\BB P\pare{\exists k\neq \pm 1:\sup_{t\in [0,T]}\norm{\int_0^t e^{\pare{t-s} A^{\pare{k}}}\ul M_\gamma^{\pare{k}}\pare{ds}}> k^2 \gamma^{\frac{1}{2}-\delta}}\leq C\gamma^{2\delta},
\end{align}
which follows from
\begin{align}\label{mkj}
\begin{aligned}
\sum_{k\neq\pm1}P\pare{\sup_{t\in [0,T]}\norm{\int_0^t e^{\pare{t-s} A^{\pare{k}}}\ul M_\gamma^{\pare{k}}\pare{ds}}> k^2 \gamma^{\frac{1}{2}-\delta}}\leq C\gamma^{2\delta}.
\end{aligned}
\end{align}
Observe that by Doob's inequality and Ito's isometry we get the following 
\begin{align}\label{pol}
\begin{aligned}
\BB P\pare{\sup_{t\in[0,T]}\abs{\int_0^tB^{\pare{k}}_{i,j}\pare{t-s}\# M_{\gamma,l}^{\pare{k}}\pare{ds}}>\varepsilon}&\le  \varepsilon^{-2}\BB E\corch{\pare{\int_0^tB^{\pare{k}}_{i,j}\pare{t-s}\# M_{\gamma,l}^{\pare{k}}\pare{ds}}^2}\\[5pt]
&\leq \varepsilon^{-2}\BB E\pare{\int_0^TB^{\pare{k}}_{i,j}\pare{t-s}^2\pic{\# M_{\gamma,1}^{\pare{k}}}\pare{ds}}
\end{aligned} 
\end{align}
for every $\varepsilon>0$, $i,j,l\in\llav{1,2}$, and $\#\in\llav{\mathfrak{R},\mathfrak{I}}$. By Lemma 5.1 in the Appendix A of \cite{KL99}, through easy computations, we get that there exists a constant $C_1$ such that 
\begin{align}\label{polk}
\pic{\# M_{\gamma,l}^{\pare{k}}}\pare{t}=\int_0^tL_\gamma\pic{\sigma_{\gamma, l}\pare{s}, \# F^{\pare{k}}}-2\pic{\sigma_{\gamma, l}\pare{s}, \# F^{\pare{k}}}L_\gamma\pic{\sigma_{\gamma, l}\pare{s}, \# F^{\pare{k}}}\ds\leq C_1\gamma t
\end{align}
As the maximum of the modulus of the entries of a matrix defines a norm, and as all the norms are equivalent, Lemma  \ref{az} guarantees the existence of a constant $C_2$ such that 
\begin{align}\label{voi}
\abs{B^{\pare{k}}_{i,j}\pare{t-s}}\le C_2e^{\frac{1}{2}\mathfrak R\pare{\mu_1^{\pare{k}}}\pare{t-s}}
\end{align}
 Plugging the estimations \eqref{polk} and \eqref{voi} into \eqref{pol}, we get that, 
\begin{align}\label{pol5}
\BB P\pare{\sup_{t\in[0,T]}\abs{\int_0^tB^{\pare{k}}_{i,j}\pare{t-s}\# M_{\gamma,l}^{\pare{k}}\pare{ds}}>\varepsilon}\leq \varepsilon^{-2}\gamma C_2^2\int_0^T e^{\mathfrak R\pare{\mu_1^{\pare{k}}}\pare{t-s}}ds\leq \frac{C_2^2}{\mathfrak R\pare{\mu_1^{\pare{k}}}}\varepsilon^{-2}\gamma 
\end{align}
for every $\varepsilon>0$, $i,j,l\in\llav{1,2}$, and $\#\in\llav{\mathfrak{R},\mathfrak{I}}$.
By \eqref{019}, we get that
\begin{align}
\mu_1^{\pare{k}}=-1+\frac{\alpha_1\phi_1^{\pare{k}}+\alpha_2\phi_2^{\pare{k}}}{2}+\frac{1}{2}\sqrt{\pare{\alpha_1\phi_1^{\pare{k}}-\alpha_2\phi_2^{\pare{k}}}^2-\tanh\pare{\beta_1\lambda}\tanh\pare{\beta_2\lambda}}
\end{align} 
 Since $\lim_{\abs{k}\to\infty}\mathfrak{R}\pare{\mu_1^{\pare{k}}}=-1$, the family $\llav{\mathfrak R\pare{\mu_1^{\pare{k}}}}_k$ is  uniformly lower bounded in $k$ and we can find a common constant $C_3$ which can replace $\frac{C_2^2}{\mathfrak R\pare{\mu_1^{\pare{k}}}}$ in the right hand side of \eqref{pol5}. Then, by decomposing $\int_0^t e^{\pare{t-s} A^{\pare{k}}}\ul M_\gamma^{\pare{k}}\pare{ds}$ first into the two coordinates and after into their real and immaginary parts, we can conclude that, for all $\zeta>0$ and for all $k\neq\pm 1$
\begin{align}\label{polok}
\begin{aligned}
\BB P\pare{\sup_{t\in[0,T]}\norm{\int_0^t e^{\pare{t-s} A^{\pare{k}}}\ul M_\gamma^{\pare{k}}\pare{ds}}>\zeta}&\le  C_4\zeta^{-2}\gamma
\end{aligned}
\end{align}
\eqref{mkj} follows after taking $\zeta=k^2\gamma^{\frac{1}{2}-\delta}$ into \eqref{polok}.  The proof for \eqref{4} is similiar, so we will omit it. We proceed with the proof  of \eqref{rew}. 

As before, it is enough to prove that 
\begin{align}\label{88}
\BB P\pare{\abs{\#{X}_{\gamma, i}^{\pare{1}}}\geq\tilde\varepsilon}\leq \tilde\varepsilon^{-2}\gamma
\end{align}
for every $\tilde \varepsilon>0$, $i\in\llav{1,2}$, and $\#\in\llav{\mathfrak{R},\mathfrak{I}}$; \eqref{88} is a consequence of Chebyshev inequality.
\end{proof}

{\bf Acknowledgments.}
It is a great pleasure to thank Errico Presutti for suggesting us the problem and for his continuous advising.
We also acknowledge (in alphabetical order) fruitful discussions with In\'es Armend\'ariz,  Anna De Masi, Pablo Ferrari, Ellen Saada, Livio Triolo, and Maria Eul\'alia Vares.
The authors also acknowledge the hospitality of Laboratoire MAP5 at Universit\'e Paris Descartes.

\bibliographystyle{amsalpha}
\bibliography{biblio}

\end{document}